\documentclass{amsart}[10pt]
\usepackage[utf8]{inputenc}
\usepackage[english]{babel}
\usepackage[T1]{fontenc}
\usepackage{amsmath}
\usepackage{amsfonts}
\usepackage{amssymb}
\usepackage{amsthm}
\usepackage{amscd}
\usepackage{geometry}
\usepackage{enumitem} 
\usepackage{cancel} 
\usepackage{graphicx}
\usepackage{mathtools}
\usepackage{color}
\usepackage{empheq}
\usepackage{hyperref}
\usepackage{comment}
\usepackage{mathrsfs}  

\newcommand\N{\mathbb{N}}
\newcommand\T{\mathbb{T}}

\newcommand\R{\mathbb{R}}

\newcommand\Div{\mathrm{div}}
\newcommand\D{\mathrm{D}}
\newcommand\eps{\varepsilon}

\renewcommand{\Im}{\operatorname{Im}}

\newcommand\reg{\mathrm{reg}}
\newcommand\spn{\mathrm{span}}
\newcommand\drag{\mathrm{drag}}

\newenvironment{proofof}[2]{\paragraph{\textit{ Proof of #1 #2.}}}{\hfill$\square$}

\theoremstyle{plain}
\newtheorem{theorem}{Theorem} [section]
\newtheorem{lemma}[theorem]{Lemma}

\newtheorem{proposition}[theorem]{Proposition}

\theoremstyle{remark}
\newtheorem{remark}[theorem]{Remark}

\theoremstyle{definition}
\newtheorem{definition}[theorem]{Definition}

\numberwithin{equation}{section}

\title[Global dissipative solutions of the defocusing isothermal ELK equations]{Global dissipative solutions of the defocusing isothermal Euler-Langevin-Korteweg equations}

\author{Quentin Chauleur}
\address{Univ Rennes, CNRS, IRMAR - UMR 6625, F-35000 Rennes, France}
\email{quentin.chauleur@ens-rennes.fr}

\begin{document}

\maketitle

\begin{abstract}
We construct global dissipative solutions on the torus of dimension at most three of the defocusing isothermal Euler-Langevin-Korteweg system, which corresponds to the Euler-Korteweg system of compressible quantum fluids with an isothermal pressure law and a linear drag term with respect to the velocity. In particular, the isothermal feature prevents the energy and the BD-entropy from being positive. Adapting standard approximation arguments we first show the existence of global weak solutions to the defocusing isothermal Navier-Stokes-Langevin-Korteweg system. Introducing a relative entropy function satisfying a Gronwall-type inequality we then perform the inviscid limit to obtain the existence of dissipative solutions of the Euler-Langevin-Korteweg system. 
\end{abstract}

\tableofcontents

\section{Introduction}
The aim of this paper is to give a proper notion of solution and to prove the global existence of such solutions of the isothermal Euler-Langevin-Korteweg system of equations (denoted ELK in the following for reader convenience):
\begin{subequations} \label{ELK}
\begin{empheq}[left=\empheqlbrace]{align}
& \partial_t \rho + \Div (\rho u) =0, \label{continuityELK}  \\
& \partial_t (\rho u) + \Div \left( \rho u \otimes u \right) + \lambda \nabla \rho + \mu \rho u= \frac{\hbar^2 }{2} \rho \nabla \left( \frac{\Delta \sqrt{\rho}}{\sqrt{\rho}} \right), \label{fluidELK} 
\end{empheq} 
\end{subequations}
where $t \in \left[0, T \right[$ for a fixed $T >0$. The unknown functions  are the density $\rho : \left[0, T \right[ \times \T^d \rightarrow \R_+ $ and the velocity field $u : \left[0, T \right[ \times \T^d \rightarrow \R^d $ of the fluid, with the initial condition $\rho(0,x)=\rho_0(x)$ and $u(0,x)=u_0(x)$ for all $x \in \T^d$, where $\T^d$ denotes the $d$-dimensional torus with $1\leq d \leq 3$. We also denote $\lambda >0$ the pressure constant, $\mu > 0$ the dissipation constant and $\hbar >0$ the renormalized Planck constant. Using the Madelung transform  $\psi = \sqrt{\rho}e^{iS/\eps}$, or in a more rigorous way the change of unknown $\rho = |\psi|^2$ and $\rho u=\hbar \Im (\psi^* \nabla \psi )$, this system is directly linked with the Schrödinger-Langevin equation: 
\begin{equation}  i \hbar \partial_t \psi + \frac{\hbar^2}{2}  \Delta \psi = \lambda \psi \log(|\psi|^2) + \frac{\hbar}{2i}  \mu  \psi \log\left( \frac{ \psi}{\psi^*} \right).  \label{NMeq} \end{equation} 

This equation first appears in Nassar's paper \cite{nassar1985} as a possible way to give a stochastic interpretation of quantum mechanics in the context of Bohmian mechanics. It had a recent renewed interest in the physics community, in particular in quantum mechanics in order to describe the continuous measurement of the position of a quantum particle (see for example \cite{nassar}, \cite{zander} or \cite{mousavi2019}) and in cosmology and statistical mechanics (see \cite{chavanis2017}, \cite{chavanis2019cosmo} or \cite{chavanis2019stat}). Note that in its physical interpretation, $\lambda=2 k_B \tau/ \hbar$ corresponds to a quantum friction coefficient, so both positive and negative signs could be of interest ($k_B$ and $\hbar$ denotes respectively the Boltzmann and the normalized Planck constant, and $\tau$ is an effective temperature), unlike the real friction coefficient $\mu$ which is taken positive (see \cite{chavanis2017}). In the mathematics community, this equation has not seen much interest yet, despite the presence of unusual nonlinear effects. In \cite{chauleur2020}, the author has shown that in the focusing case ($\lambda < 0$) every well-prepared density function of the solution of equation \eqref{NMeq} on $\R^d$  converges to a Gaussian function weakly in $L^1(\R^d)$, whose mass and center are uniquely determined by its initial data $\rho_0$ and $u_0$, whereas in the defocusing case ($\lambda > 0$) every solution disperses to 0, with a slower dispersion rate than usual directly affected by the nonlinear Langevin potential $\frac{1}{2i}\log\left( \psi/\psi^* \right)$. Still in the defocusing case, up to a space-time rescaling incorporating dispersive effects, the density $\rho$ of every rescaled solution of equation \eqref{NMeq} on $\R^d$ also converges to a Gaussian function weakly in $L^1(\R^d)$, a phenomenon which is reminiscent of the defocusing logarithmic Schrödinger equation \cite{carles2018} (corresponding to the case $\mu =0$ in \eqref{NMeq}). \\ 

The question of the existence of solutions to this kind of quantum system is already dealt with in the case of barotropic pressure of the form $P(\rho)= \lambda \rho^{\gamma}$ in \cite{antonelli2009}, where $\gamma >1$ (and $\mu=0$). However, the proof is based on the link with the power-like Schrödinger equation 
\begin{equation} \label{powerNLS}
i \hbar \partial_t \psi + \frac{\hbar^2}{2}  \Delta \psi = \lambda \psi |\psi|^{\gamma-1}  
\end{equation}
and the use of Strichartz estimates which do not seem to be helpful for the logarithmic nonlinearity. In fact, rather than estimate $\nabla( \psi |\psi^{\gamma-1}| )$, we have to deal with the quantity $\nabla( \psi \log (|\psi|^2))$ which becomes unbounded when the wave function $\psi$ vanishes to 0, preventing us from applying the fractional steps method of \cite{antonelli2009}. In the second part of \cite{carles2019}, the authors have shown the existence of solutions to the isothermal Euler-Korteweg system (which corresponds to the case $\mu=0$ in the system \eqref{ELK}), but again their proof is strongly based on the link with the logarithmic Schrödinger equation, which does not seem to be useful in our case due to the ill-posed nonlinear Langevin potential $\frac{1}{2i}\log\left( \psi/\psi^* \right)$. In a more regular framework, note that in the recent work \cite{ferriere2020}, the author shows the local existence in time of solutions to the Euler-Korteweg system satisfying some analytic regularity, which may be extended to the Euler-Langevin-Korteweg equations. \\

Unlike the Navier-Stokes-Korteweg system of equations (which corresponds to system \eqref{ELK} adding the viscous term $\nu \Div \left( \rho \mathbb{D} u \right)$ on the right hand side of \eqref{fluidELK}, see \cite{jungel2010} or the system \eqref{NSLK} below), the lack of viscous term in the Euler-Korteweg system prevents us from adding some smoothing terms in the equations and then make these terms tend to 0. However, in the recent paper \cite{Bresch2019}, the authors manage to pass to the viscous limit $\nu \rightarrow 0$ for a specific notion of weak solution called dissipative solution. We will base our proof of existence on this approach, with the specificity of taking an isothermal pressure law $\lambda \rho$ (which leads to the use of energy with no definite sign, see below) and adding a dissipation term $\mu \rho u$. Note that in the recent paper \cite{bresch2020} the authors have shown the exponential decay to equilibrium of global weak solutions of the Navier-Stokes-Korteweg system with such a dissipation term and for both barotropic and isothermal pressure laws. \\

Dissipative solutions were introduced by Lions \cite{lions1996}, and were recently used in fluid models in order to give a rigorous justification of some viscous singular limits (see \cite{feireisl2014}). They are based on the use of some particular functionals called \textit{relative entropies}, namely
\[  \mathcal{E} (U|V) : X \times Y \mapsto \R_+,\]
where $X$ denotes a Banach space of \textit{weak solutions} $U$, and $Y \subset X$ denotes a Banach space of  \textit{strong solutions} $V$ embedded into our space of weak solutions of the fluid system
\[ \frac{d}{dt} U(t) = \mathcal{A}(t,U(t)),\ t>0, \ U(0)=U_0, \]
where $\mathcal{A}$ is a (nonlinear) generator. A relative entropy has to enjoy the following three properties:
\begin{itemize}
    \item \textbf{Distance property.} For every $(U,V) \in X \times Y$,
    \[  \mathcal{E} (U|V) \geq 0 \ \ \  \text{and} \ \ \ \mathcal{E} (U|V)=0 \ \text{only if} \ U=V. \]
     \item \textbf{Lyapounov functional.} If $V$ is an \textit{equilibrium solution}, namely $\mathcal{A}(t,V(t))=0$ for all $t\geq 0$, then $V \in Y$ and 
     \[ \frac{d}{dt}  \mathcal{E} (U|V) \leq 0. \]
     \item \textbf{Gronwall inequality.} For every $(U,V) \in X \times Y$, for a.a. $t\geq 0$,
     \[ \mathcal{E} (U|V)(t) - \mathcal{E} (U|V)(0) \leq \int_0^t \mathcal{E} (U|V)(s) ds.    \]
\end{itemize}

The presence of the nonlinear quantum viscous term $\hbar^2 \rho \nabla \left( \Delta \sqrt{\rho}/\sqrt{\rho} \right)$, usually called Bohm potential, makes difficult to define a relative entropy for the ELK system \eqref{ELK}. In order to circumvent this difficulty, we denote $\overline{v}= \frac{\hbar}{2} \nabla \log \rho$ and introduce the augmented ELK system:
\begin{subequations} \label{ELK_augmented}
\begin{empheq}[left=\empheqlbrace]{align}
& \partial_t \rho + \Div (\rho u) =0, \label{continuityELK_augmented}  \\
& \partial_t (\rho u) + \Div \left( \rho u \otimes u \right) + \lambda \nabla \rho + \mu \rho u= \frac{\hbar}{2} \Div( \rho \nabla \overline{v}) , \label{fluidELK_augmented} \\
& \partial_t (\rho \overline{v}) + \Div \left( \rho \overline{v} \otimes u \right) +\frac{\hbar}{2} \Div \left( \rho \nabla u^\top \right) =0. \label{fluid2ELK_augmented} 
\end{empheq} 
\end{subequations}
Note that this augmented formulation was first introduced in \cite{bresch2015} as a numerical tool in order to compute quantum fluid systems, and is based on the identity
\[ \frac{\hbar^2}{2} \rho \nabla \left(  \frac{\Delta \sqrt{\rho}}{\sqrt{\rho}} \right) = \frac{\hbar^2}{4} \Div ( \rho \nabla^2 \log \rho ).\]
Formally differentiating in space equation \eqref{continuityELK}, we see that a solution of the ELK system \eqref{ELK} also stands as a solution of \eqref{ELK_augmented} (see Remark \ref{weak_to_augmented_remark} below). We also define the functional $ H(\rho)= \rho \log \rho - \rho$, so we can introduce the following relative entropy entropy functional for our system:
\[  \mathcal{E}_{ELK}(\rho, u,\overline{v} | R,U,\overline{V})(t)
= \frac{1}{2} \int_{\T^d} \rho (|\overline{v}-\overline{V}|^2  +|u-U|^2 )+ \lambda \int_{\T^d} H(\rho|R) + \mu \int_0^t \int_{\T^d}  \rho |u-U|^2,  \]
where 
\[H(\rho |R) := H(\rho) - H(R) - H'(R) (\rho - R) .\]
Note that here $(\rho,u,\overline{v})$ stands as a weak solution of \eqref{ELK_augmented}, where $(R,U,\overline{V})$ has to be seen as a strong solution of this system.  We will also use the convention
\[ \mathcal{E}_{ELK}(\rho, u,\overline{v}) := \mathcal{E}_{ELK}(\rho, u,\overline{v} | 1,0,0). \]

We now introduce the concept of dissipative solution, which is induced by the Gronwall inequality of our relative entropy functional as follows:
\begin{definition}
Let $T>0$ and $(\rho_0,u_0,\overline{v}_0= \frac{\hbar}{2} \nabla \log \rho_0)$ such that $\mathcal{E}_{ELK}(\rho_0,u_0,\overline{v}_0) < \infty$. We say that $(\rho,u,\overline{v})$ is a \textbf{dissipative solution} of the augmented ELK system \eqref{ELK_augmented} in $\left[ 0, T \right[ \times \T^d$ with initial data $(\rho_0,u_0,\overline{v}_0)$ if the following holds:
\[  \mathcal{E}_{ELK}(\rho, u,\overline{v} | R,U,\overline{V})(t) \leq \mathcal{E}_{ELK}(\rho, u,\overline{v} | R,U,\overline{V})(0) e^{Ct} + b_{ELK}(t) + C \int_0^t b_{ELK}(s) e^{C(t-s)} ds,  \]
for a.e. $t \in \left[ 0, T \right[$, where 
\[ b_{ELK}(t)= \int_0^t \int_{\T^d} \frac{\rho}{R} | \mathscr{E}(R,U) \cdot (U-u)|   \]
for all smooth function $U$ and $(R,\overline{V}, \mathscr{E})$ defined respectively through the initial condition
\[\int_{\T^d} R_0 = \int_{\T^d} \rho_0,  \]
and the system
\begin{subequations} \label{ELK_augmented_def}
\begin{empheq}{align}
& \partial_t R + \Div (R U) =0,   \label{continuityELK_augmented_def}  \\
& \overline{V}= \frac{\hbar}{2} \nabla \log R, \label{V_augmented_def}  \\
& \mathscr{E}(R,U)=R(\partial_t U + U \cdot \nabla U) + \lambda \nabla R + \mu R U - \frac{\hbar}{2} \Div ( R \nabla \overline{V}) , \label{fluidELK_augmented_def} 
\end{empheq} 
\end{subequations}
and $C=C(\hbar,R,U,\overline{V})$ is a constant uniformly bounded on $\R_+ \times \T^d$.
\label{dissipative_solution}
\end{definition}

\begin{remark} \label{weak_to_augmented_remark}
As we are unable to have the existence of strong solutions to the system \eqref{ELK_augmented}, we have to introduce the function $\mathscr{E}(R,U)$ which has to be understood as an error function. In fact, if $\mathscr{E}(R,U)=0$ then we recover a strong solution of \eqref{ELK_augmented} through the system \eqref{ELK_augmented_def} by differentiating in space the mass equation \eqref{continuityELK_augmented_def}, which gives
\[ \partial_t \nabla R + \nabla \Div ( R U)=\partial_t \nabla R + \Div ( \nabla(R U)^\top)=0,\]
that could be written
\[ \partial_t \nabla R + \Div ( R \nabla \log R \otimes U ) + \Div( \nabla (R U)^\top - R \nabla \log R \otimes U)=0   \]
in order to show that $\overline{V}$ fulfills equation \eqref{fluid2ELK_augmented}. Of course, we have to take into account this error term in the Gronwall inequality of Definition \ref{dissipative_solution} through the function $b_{ELK}$, whose expression will make sense with the calculations of the proof of Proposition \ref{secund_ineq}. 
\end{remark}

\begin{remark} \label{Csizar-Kullback}
We have to state at this stage that the positivity of our entropy functional $\mathcal{E}_{ELK}$ is unclear due to the presence of the logarithmic function in the isothermal contribution $\lambda \int_{\T^d} H(\rho|R)$. In order to get positivity, we have to invoke the Csiszár-Kullback inequality (see e.g. \cite{ineg_sobo_log}) which tells that for $f,g \geq 0$ such that $\int f = \int g$, we have
\[ \| f-g \|^2_{L^1(\T^d)} \leq 2 \| f \|_{L^1(\T^d)} \int_{\T^d} f(x) \log \left( \frac{f(x)}{g(x)} \right) dx.  \]
Recalling that $\int r_0 = \int \rho_0$ and using the continuity equations \eqref{continuityELK_augmented_def} and \eqref{continuityELK_augmented} that ensures that $\forall t \in \left[0,T \right[$, $\| \rho(t) \|_{L^1(\T^d)} = \| R(t) \|_{L^1(\T^d)}$, we get that 
\[ \int_{\T^d} H(\rho |R) = \int_{\T^d} \left( \rho \log \rho  - \rho -( R \log R -R) - \log R (\rho -R) \right) = \int_{\T^d} \rho \log \left( \frac{\rho}{R} \right) \geq 0.  \]
Under the assumption $\lambda >0$, we then recover the positivity of $\mathcal{E}_{ELK}$, which will be crucial in the proof of Theorem \ref{gronwall_strong}. Note that in the focusing case $\lambda <0$, Csiszár-Kullback inequality does not ensure positivity on our relative entropy anymore, hence we cannot define a proper notion of dissipative solution in the focusing case and apply the results from Section 3. However, all the results from Section 2 are still valid in the case $\lambda <0$ (see Remark \ref{remark_focusing_case} below).
\end{remark}

We can now state the main result of this paper:
\begin{theorem} \textbf{(Global existence for the augmented isothermal ELK).} \\
Let $T>0$ and $(\rho_0,u_0,\overline{v}_0= \frac{\hbar}{2} \nabla \log \rho_0)$ such that $\mathcal{E}_{ELK}(\rho_0,u_0,\overline{v}_0) < \infty$, then there exists a dissipative solution of the augmented ELK system \eqref{ELK_augmented} in $\left[ 0, T \right[ \times \T^d$ with initial data $(\rho_0,u_0,\overline{v}_0)$.
\label{main_theo}
\end{theorem}
The rest of this paper is devoted to the proof of Theorem \ref{main_theo}, and is organized as follows. In Section 2, we introduce several regularizing viscous terms in the ELK system which enable us to construct a solution of our regularized system. Letting these several viscosity terms one by one go to zero, we recover the existence of a weak solution to the Navier-Stokes-Langevin-Korteweg equation, which corresponds to the ELK system with an additional term $\nu \Div \left( \rho \mathbb{D} u \right)$ in the momentum equation. Section 3 is then dedicated to the inviscid limit $\nu \rightarrow 0$, where we have to introduce augmented formulation for both ELK and NSLK systems. Finally we give in Appendix the definition of the operators and some technical lemmas which will be used all along this paper.

\section{The isothermal Navier-Stokes-Langevin-Korteweg system}

We look at the Navier-Stokes-Langevin-Korteweg system (denoted NSLK in the following for reader convenience), namely:
\begin{subequations} \label{NSLK}
\begin{empheq}[left=\empheqlbrace]{align}
& \partial_t \rho + \Div (\rho u) =0, \label{continuityNSLK}  \\
& \partial_t (\rho u) + \Div \left( \rho u \otimes u \right) + \lambda \nabla \rho + \mu \rho u= \frac{\hbar^2 }{2} \rho \nabla \left( \frac{\Delta \sqrt{\rho}}{\sqrt{\rho}} \right) + \nu \Div \left( \rho \mathbb{D} u \right), \label{fluidNSLK} 
\end{empheq} 
\end{subequations}
where $\mathbb{D} u = ( \nabla u + \nabla u ^\top)/2$. This system formally enjoys the following energy estimate:
\begin{equation}  E_{NSLK}(\rho,u)(t) +  \int_0^t D_{NSLK}(\rho,u)(s) ds \leq E_0, \label{eq_energy_NSLK} \end{equation}
with
\begin{gather}
    E_{NSLK}(\rho,u)=\frac{1}{2} \int_{\T^d} \left(  \rho |u|^2  +\hbar^2 |\nabla \sqrt{\rho} |^2 \right) + \lambda \int_{\T^d} H(\rho) , \label{energy_NSLK} \\ 
    D_{NSLK}(\rho,u) = \mu \int_{\T^d} \rho |u|^2  + \nu \int_{\T^d} \rho | \mathbb{D} u |^2. \label{dissipation_NSLK}
\end{gather}
\begin{remark} \label{focusing_remark}
Similarly to Remark \ref{Csizar-Kullback}, in order to use \eqref{eq_energy_NSLK} to get some regular bounds on our solutions, we need to show that $E_{NSLK}$ is also bounded from below, which is not obvious due to the presence of the logarithmic term $\lambda \int_{\T^d} H(\rho)$. Denoting $E(t)+ \lambda \| \rho_0 \|_{L^1}=E^+(t)-E^-(t)$, with
\begin{gather*}
    E^+(t)=\int_{\T^d} \left( \frac{1}{2} \rho |u|^2 +  \hbar^2 |\nabla \sqrt{\rho} |^2  \right) dx + \lambda \int_{\rho > 1} \rho \log \rho \geq 0,\\
    E^-(t) = -\lambda \int_{\rho < 1} \rho \log \rho = \lambda \int_{\rho < 1} \rho \log\left(\frac{1}{\rho} \right)\geq 0,
\end{gather*}
we remark that $E^-(t)$ is controlled by
\[ \int_{\rho < 1} \rho \log \left(\frac{1}{\rho} \right) \lesssim \int_{\T^d} \rho^{1-\eps} \leq C \| \rho_0 \|_{L^1}^{1-\eps} \]
by the Hölder inequality on the compact set $\T^d$ with $\eps >0$ small enough, hence $E^-(t)$ and $E^+(t) \in L^{\infty}$, so the energy $E_{NSLK}$ is indeed bounded from below.
\end{remark}

The BD-entropy of a quantum fluid system, firstly introduced in \cite{bresch2003} and \cite{bresch2004}, is now a classical tool in order to get further regularity bounds on our solution. This system also enjoys a BD-entropy estimate:
\begin{equation}  \mathcal{E}_{NSLK}(\rho,u)(t) +  \int_0^t \mathcal{D}_{NSLK}(\rho,u)(s) ds \leq \mathcal{E}_0, \label{eq_BD_entropy_NSLK} \end{equation}
with
\begin{gather}
    \mathcal{E}_{NSLK}(\rho,u)=\frac{1}{2} \int_{\T^d} \left(  \rho |u+\frac{\nu}{2} \nabla \log \rho|^2  +\hbar^2 |\nabla \sqrt{\rho} |^2 \right) + \lambda' \int_{\T^d} H(\rho) , \label{energy_BD_entropy_NSLK} \\ 
    \mathcal{D}_{NSLK}(\rho,u) = \mu \int_{\T^d} \rho |u|^2  + \frac{\nu}{2} \int_{\T^d} \rho | \mathbb{A} u |^2 + \frac{\nu \hbar^2}{2}  \int_{\T^d} \rho | \nabla^2 \log \rho |^2 + 2 \nu  \int_{\T^d}|\nabla \sqrt{\rho} |^2, \label{dissipation_BD_entropy_NSLK}
\end{gather}
where $\lambda':=\lambda - \mu \nu$ and $\mathbb{A}u= (\nabla u - \nabla u^\top)/2$. We now give the notion of weak solution induced by these quantities:
\begin{definition} 
Let $T >0$ and $(\rho_0,u_0)$ such that $E_{NSLK}(\rho_0,u_0) < \infty$. We say that $(\rho,u)$ is a \textbf{weak solution} of the NSLK system \eqref{NSLK} in $\left[ 0,T \right[ \times \T^d$ with initial data $(\rho_0,u_0)$, if there exists locally integrable functions $\sqrt{\rho}$, $\sqrt{\rho}u $ such that, by defining $\rho := \sqrt{\rho}^2$ and $\rho u:= \sqrt{\rho} \sqrt{\rho}u$, the following holds:

\begin{enumerate}[label=(\roman*)]

    \item The global regularity:
    \[ \sqrt{\rho} \in L^{\infty}(\left[ 0,T \right[;H^1(\T^d)), \ \ \ 
      \sqrt{\rho} u  \in L^{\infty}(\left[ 0,T \right[;L^2(\T^d)),  \]
      with the compatibility condition
      \[ \sqrt{\rho} \geq 0 \text{ a.e. on } (0,\infty) \times \T^d, \ \ \ 
      \sqrt{\rho} u =0 \text{ a.e. on } \left\{ \rho=0 \right\}. \]
      
      \item For any test function $\eta \in \mathcal{C}_0^{\infty}(\left[ 0,T \right[ \times \T^d)$, 
      
      \begin{equation} \label{continuity_weak_NSLK}      
      \int_0^T \int_{\T^d} (\rho \partial_t \eta + \rho u \cdot \nabla \eta )dx dt + \int_{\T^d} \rho_0 \eta(0) dx =0, \end{equation} 
      and for any test function $\zeta \in \mathcal{C}_0^{\infty}(\left[ 0,T \right[ \times \T^d; \T^d)$,
      \begin{multline} \int_0^T \int_{\T^d} \left( \rho u \cdot \partial_t \zeta + \sqrt{\rho} u \otimes \sqrt{\rho} u : \nabla \zeta +\lambda \rho \Div (\zeta) - \mu \rho u \cdot \zeta     + \hbar^2  \nabla \sqrt{\rho} \otimes \nabla \sqrt{\rho} : \nabla \zeta \right. \\ \left. - \frac{\hbar^2 }{4} \rho \Delta \Div \zeta- \nu \sqrt{\rho} \mathbb{S}_N(u) \cdot \nabla \zeta \right) dx dt  
      +  \int_{\T^d} \rho_0 u_0 \zeta(0) dx =0, \label{fluid_weak_NSLK} \end{multline}
      with $ \mathbb{S}_N(u)= (\mathbb{T}_N(u) + \mathbb{T}_N(u)^\top)/2$, where $\mathbb{T}_N(u) $ is defined through the compatibility condition
      \[ \sqrt{\rho} \ \mathbb{T}_N(u) = \nabla( \sqrt{\rho} \sqrt{\rho} u ) - 2  \sqrt{\rho} u \otimes \nabla \sqrt{\rho} .\]      
      
      \item  For any test function $\xi \in \mathcal{C}_0^{\infty}(\T^d)$, 
      \[ \lim\limits_{\substack{t \to 0 }} \int_{\T^d} \rho(t,x) \xi(x) dx =  \int_{\T^d} \rho_0(t,x) \xi(x) dx,\]
      \[ \lim\limits_{\substack{t \to 0 }} \int_{\T^d} \rho(t,x) u(t,x) \xi(x) dx =  \int_{\T^d} \rho_0(t,x) u_0(t,x) \xi(x) dx.\]
\end{enumerate}
\label{weak_NSLK}
\end{definition}
\begin{remark}
The tensor valued function $\mathbb{T}_N(u)$ has to be understood as $\sqrt{\rho} \nabla u$. Of course, we can not define this term properly, so we use the algebraic identity given in the previous definition, which makes sense in the distribution sense in view of the regularity assumption.
\end{remark}

\begin{remark}
 Note that we have used in the momentum equation \eqref{fluid_weak_NSLK} the identity
 \[ \frac{\hbar^2}{2} \rho \left( \frac{\Delta \sqrt{\rho}}{\sqrt{\rho}} \right) = \frac{\hbar^2}{4} \Delta \nabla \rho - \hbar^2 \Div ( \nabla \sqrt{\rho} \otimes  \nabla \sqrt{\rho} ).   \]
\end{remark}

The aim of this section is now to prove the following theorem:
\begin{theorem} \textbf{(Global existence for the isothermal NSLK).} \label{exsitence_NSLK}\\
Let $T >0$, $\nu >0$ and $(\rho_0,u_0)$ such that $E_{NSLK}(\rho_0,u_0) < \infty$, then there exists a weak solution of system \eqref{continuityNSLK}-\eqref{fluidNSLK} in $\left[ 0,T \right[ \times \T^d$ with initial data $(\rho_0,u_0)$. Furthermore, for almost every $t \in \left[0,T\right[$, equations \eqref{eq_energy_NSLK} and \eqref{eq_BD_entropy_NSLK} holds.
\label{prop_NSLK}
\end{theorem}
\begin{remark} \label{remark_focusing_case}
Note that Theorem \ref{exsitence_NSLK} still holds in the focusing case $\lambda <0$. In fact, to adapt the proof, we just have to show that $E_{NSLK}$ is bounded from below similarly to Remark \ref{focusing_remark}. Writing $E_{NSLK}(t)=E^+(t)-E^-(t)$, where
\begin{gather*}
    E^+(t)=\int_{\T^d} \left( \frac{1}{2} \rho |u|^2 +  \hbar^2 |\nabla \sqrt{\rho} |^2 - \lambda \rho \right) dx + \lambda \int_{\rho < 1} \rho \log \rho \geq 0,\\
    E^-(t) = -\lambda \int_{\rho > 1} \rho \log \rho \geq 0,
\end{gather*}
and using the conservation of mass $\| \rho(t) \|_{L^1}=\| \rho_0 \|_{L^1}$, we get that
\[ \int_{\rho > 1} \rho \log \rho \lesssim \int_{\T^d} \sqrt{\rho}^{2+\eps} \leq \| \sqrt{\rho} \|_{L^2}^{(2+\eps)(1-\alpha)}  \| \nabla \sqrt{\rho} \|_{L^2}^{(2+\eps)\alpha} =  C \| \nabla \sqrt{\rho} \|_{L^2}^{(2+\eps)\alpha} \]
by the Gagliardo-Nirenberg inequality with $\alpha=1/2-1/(2+\eps)$, for $\eps >0$ small enough (see Lemma \ref{gagliardo_nirenberg}). Finally,
\[ E^+(t) \leq E_0 +E^-(t) \leq E_0 + C E^+(t)^{\delta} \]
with $\delta <1$, inducing that $E^+(t) \in L^{\infty}$ and so $E^-(t)$. Unfortunately, we are unable to define a positive relative entropy functional in the focusing case, so we can't make the viscous limit $\nu \rightarrow 0$ in this case.\\
\end{remark}

\subsection{Regularized NSLK system}
Following \cite{vasseur2016}, we first aim at proving the existence of a solution to the following regularized NSLK system:
\begin{subequations} \label{regularized_NSLK}
\begin{empheq}[left=\empheqlbrace]{align}
& \partial_t \rho + \Div (\rho u) = \delta_1 \Delta \rho, \label{continuity_regularized_NSLK}  \\
&  \partial_t (\rho u) + \Div \left( \rho u \otimes u \right) + \lambda \nabla \rho + \mu \rho u + r_0 u + r_1 \rho |u|^2u = \frac{\hbar^2 }{2} \rho \nabla \left( \frac{\Delta \sqrt{\rho}}{\sqrt{\rho}} \right) + \nu \Div \left( \rho \mathbb{D} u \right)  \\ 
&  + \delta_2 \Delta^2 u - \delta_1 (\nabla \rho \cdot \nabla) u + \eta_1 \nabla \rho^{-\alpha} + \eta_2 \rho \nabla \Delta^{2s+1} \rho, \notag
 \label{fluid_regularized_NSLK}
\end{empheq} 
\end{subequations}
where the regularization parameters verify $0 < \delta_1, \delta_2, \eta_1, \eta_2, r_0, r_1 <1$ and $\alpha > 0$, $s \in \mathbb{N}^*$ are chosen sufficiently large (to be fixed later on). Integrating equation \eqref{continuity_regularized_NSLK} in space, we get the conservation of mass, namely for all $t \geq 0$,
\begin{equation} \label{conservation_mass}
    \int_{\T^d} \rho(t,x) dx = \int_{\T^d} \rho_0(x) dx.
\end{equation}

Then, multiplying formally equation \eqref{fluid_regularized_NSLK} with $u$ and combining it with equation \eqref{continuity_regularized_NSLK} we get the energy estimate, for almost all $ t \geq 0$:
\begin{equation} \label{eq_energy_regularized_NSLK}
E_{\reg}(\rho,u)(t) +  \int_0^t D_{\reg}(\rho,u)(s) ds \leq E_0,  
\end{equation}
where
\begin{equation}
    E_{\reg}(\rho,u)=\frac{1}{2} \int_{\T^d} \left(  \rho |u|^2  +\hbar^2 |\nabla \sqrt{\rho} |^2 \right) + \lambda \int_{\T^d} H(\rho)  + \frac{\eta_1}{\alpha + 1} \int_{\T^d} \rho^{-\alpha} + \frac{\eta_2}{2} \int_{\T^d} | \nabla \Delta^s \rho |^2, \label{energy_regularized_NSLK} 
\end{equation}
 \begin{multline}
    D_{\reg}(\rho,u) = \mu \int_{\T^d} \rho |u|^2  + \nu \int_{\T^d} \rho | \mathbb{D} u |^2 + \delta_2 \int_{\T^d} | \Delta u|^2 + \delta_1 \eta_2 \int_{\T^d} | \Delta^{s+1} \rho|^2 \\
    + 4 \delta_1 \int_{\T^d} |\nabla \sqrt{\rho}|^2 + \frac{4 \delta_1 \eta_1}{\alpha} \int_{\T^d} |\nabla \rho^{-\alpha/2}|^2 + r_0 \int_{\T^d}|u|^2 + r_1 \int_{\T^d} \rho|u|^4 + \frac{\delta_1 \hbar^2}{2} \int_{\T^d} \rho |\nabla^2 \log \rho|^2. 
   \label{dissipation_regularized_NSLK}
 \end{multline}   
 
 \begin{definition} 
Let $T >0$. We say that $(\rho,u)$ is a \textbf{weak solution} of the regularized NSLK system \eqref{regularized_NSLK} in $\left[ 0,T \right[ \times \T^d$ with initial data $(\rho_0,u_0) \in L^1(\T^d) \times  L^2(\T^d)$, if the following holds:

\begin{enumerate}[label=(\roman*)]

    \item The global regularity:
    \[ \rho \in H^1(\left[0,T\right[; H^1(\T^d)) \cap \mathcal{C}^0 (\left[0,T\right[; H^{2s}(\T^d)) \cap L^2(\left[0,t \right[; H^{2s+2}(\T^d)), \]
    \[  1/\rho \in \mathcal{C}^0(\left[0,T\right[ \times \T^d), \]
    \[ u \in L^{\infty}(\left[0,T \right[; L^2(\T^d)) \cap L^2(\left[0,T \right[; H^2(\T^d)). \]
      
      \item For any test function $\eta \in \mathcal{C}_0^{\infty}(\left[ 0,T \right[ \times \T^d)$, 
      
      \begin{equation} \label{continuity_regularized_NSLK_eq}
      \int_0^T \int_{\T^d} (\rho \partial_t \eta + \rho u \cdot \nabla \eta +\delta_1 \rho \Delta \eta )dx dt + \int_{\T^d} \rho_0 \eta(0) dx =0, 
      \end{equation} 
      and for any test function $\zeta \in \mathcal{C}_0^{\infty}(\left[ 0,T \right[ \times \T^d; \T^d)$,
      \begin{multline} \int_0^T \int_{\T^d} \left( \rho u \cdot \partial_t \zeta + \rho u \otimes  u : \nabla \zeta +\lambda \rho \Div (\zeta) - \mu \rho u \cdot \zeta  + \hbar^2  \nabla \sqrt{\rho} \otimes \nabla \sqrt{\rho} : \nabla \zeta - \frac{\hbar^2 }{4} \rho \Delta \Div \zeta \right. \\ \left. - \nu \rho \mathbb{D} u: \nabla \zeta - r_0 u \cdot \zeta - r_1 \rho|u|^2 u \cdot \zeta  - \delta_1 \nabla u : \nabla \rho \otimes \zeta - \delta_2 \Delta u \cdot \delta \zeta - \eta_1 \rho^{-\alpha} \Div \zeta   \right. \\
      \left. - \eta_2 \Delta^{s+1} \rho \Delta^s \left[ \nabla \rho \cdot \zeta + \rho \Div \zeta   \right] \right) dx dt   +  \int_{\T^d} \rho_0 u_0 \zeta(0) dx =0.
      \label{fluid_regularized_NSLK_eq}
      \end{multline} 
         
\end{enumerate}
\label{weak_regularized_NSLK}
\end{definition}
 
 \begin{proposition} \label{regularized_solution_prop}
 Let $T>0$ and
 \begin{equation} \label{CI_regularized}
\rho_0 \in \mathcal{C}^{\infty}_0 ( \T^d), \ \ \ u_0 \in L^2(\T^d), \ \ \ \inf_{x \in \T^d} \rho_0(x) \geq \theta > 0,  
 \end{equation} 
 then there exists a weak solution $(\rho,u)$ of system \ref{regularized_NSLK} in  $\left[ 0,T \right[ \times \T^d$ with initial data $(\rho_0,u_0)$ which satisfies moreover the conservation of mass \eqref{conservation_mass} and the energy estimate \eqref{eq_energy_regularized_NSLK}.
 \end{proposition}
 
 \begin{remark} \label{regularity_regularized_NSLK}
  We note that the conservation mass \eqref{conservation_mass} and the energy estimate \eqref{energy_regularized_NSLK} induce the following uniform bounds with respect to $E(\rho_0,u_0)$:
  \begin{align*}
      \rho (1+|\log \rho|) \in L^{\infty}(\left[ 0,T \right[; L^1(\T^d)), \ \ \ \sqrt{\rho} u \in L^{\infty}(\left[ 0,T \right[; L^2(\T^d)), \\
      \nu \sqrt{\rho} \nabla u \in L^2(\left[ 0,T \right[; L^2(\T^d)), \ \ \ \hbar \nabla \sqrt{\rho} \in L^{\infty}(\left[ 0,T \right[; L^2(\T^d)), \\
      \sqrt{r_0} u \in L^2(\left[ 0,T \right[; L^2(\T^d)), \ \ \ \sqrt{r_1} \rho^{\frac{1}{4}} u \in L^4(\left[ 0,T \right[; L^4(\T^d)), \\
      \eta_1^{\frac{1}{\alpha}} \rho^{-1} \in L^{\infty}(\left[ 0,T \right[; L^{\alpha}(\T^d)), \ \ \ \sqrt{\eta_2} \rho \in  L^{\infty}(\left[ 0,T \right[; H^{s+1}(\T^d)), \\
      \sqrt{\delta_1 \eta_1} \nabla \rho^{-\frac{\alpha}{2}} \in L^2(\left[ 0,T \right[; L^2(\T^d)), \ \ \ \sqrt{\delta_1 \eta_2} \Delta^{s+1} \rho \in L^2(\left[ 0,T \right[; L^2(\T^d)), \\
      \sqrt{\delta_2} \Delta u \in L^2(\left[ 0,T \right[; L^2(\T^d)).
  \end{align*}
  Also, combining the energy estimate \eqref{eq_energy_regularized_NSLK} with Lemma \ref{technical_lemma} ensures that
  \[ \sqrt{\nu \hbar^2} \nabla^2 \sqrt{\rho}  \in L^2(\left[ 0,T \right[; L^2(\T^d)), \ \ \  (\nu \hbar^2)^{\frac{1}{4}} \nabla \rho^{\frac{1}{4}} \in L^4(\left[ 0,T \right[; L^4(\T^d)).  \]
  Finally, combining these bounds with Lemma \ref{inverse_lemma} we obtain that there exists a positive constant $C(E_{reg}(\rho_0,u_0),\eta_1,\eta_2,\theta)$ such that
  \[ \| 1/\rho \|_{L^{\infty}(\left[ 0,T \right[ \times \T^d)} \leq C(E_{reg}(\rho_0,u_0),\eta_1,\eta_2,\theta).   \]
 \end{remark}
 
 \begin{proof} In all the convergences mentioned in the following proof, we have to extract subsequences that we do not relabel for conciseness. \\
 
 \textbf{Step 1: Faedo-Galerkin approximation.} Following \cite{jungel2010} and \cite{vasseur2016}, we denote $(e_k)_{k \in \N}$ an orthonormal basis of $L^2(\T^d)$ (which is also an orthogonal basis of $H^1(\T^d)$) and we introduce  the finite-dimensional space $X_N=\spn \left\{ e_1, \ldots, e_N \right\}$. Then the classical theory of parabolic equation gives us the existence of a function
 \[\rho_N \in \mathcal{C}^{0}(\left[0,T \right[; H^{2s+1}(\T^d))\] 
 satisfying equation \eqref{continuity_regularized_NSLK_eq} on $\left[0,T \right[ \times \T^d$, and the maximum principle provides some lower and upper bounds on the density $\rho_N$. In particular, since we assumed that $\rho_0 \geq \theta >0$, we get that $\rho_N$ is strictly positive. Then, by a fixed point argument, the standard theory for systems of ordinary differential equations provides the existence of a unique classical solution
 \[ u_N \in  \mathcal{C}^{0}(\left[0,T \right[;X_N)  \]
 satisfying equation \eqref{fluid_regularized_NSLK_eq} (taking $\rho=\rho_N$) on $\left[0,T \right[ \times X_N$. Furthermore,  integrating equation \eqref{continuity_regularized_NSLK} on space, we get the conservation of mass of $\rho_N$ \eqref{conservation_mass}, and multiplying equation \eqref{fluid_regularized_NSLK} with $u_N$ and combining with equation \eqref{continuity_regularized_NSLK} we get the energy estimate \eqref{eq_energy_regularized_NSLK} for $(\rho_N,u_N)$. \\
 
  \textbf{Step 2: Uniform estimates on the approximate solutions.} We note that $\rho_N(0,.)=\rho_0$ and $\rho_Nu_N(0,.)= \mathbb{P}_N \left[ \rho_0 u_0 \right]$, where $\mathbb{P}_N$ denotes the $L^2(\T^d)$-projection onto $X_N$. In particular, since by assumption $\rho_0u_0 \in L^2(\T^d)$, we have that
  \[ E_{\reg}(\rho_N,u_N)(0) \leq E_{\reg}(\rho_0,u_0),   \]
  and so with the energy inequality \eqref{eq_energy_regularized_NSLK} satisfied by $(\rho_N,u_N)$ we get that, for all $N \in \mathbb{N}^*$,
  \begin{equation} \label{energy_inequality_faedo}
 \sup_{0 \leq t \leq T}  E_{\reg}(\rho_N,u_N)(t) + \int_0^T D_{\reg}(\rho_N,u_N)(s) ds \leq C(  E_{\reg}(\rho_0,u_0)). 
  \end{equation} 
 From this inequality we get several uniform bounds on $(\rho_N,u_N)$ with respect to $N$, in particular we get the existence of some functions $\rho$ and $V$ such that
 \begin{align*}
 \rho_N \rightharpoonup \rho \ \text{weakly in} \ L^{\infty}(\left[0,T \right[; H^{2s+1}(\T^d)), \\ \sqrt{\rho_N} u_N \rightharpoonup V \ \text{weakly in} \ L^{\infty}(\left[0,T \right[;L^2(\T^d)).
  \end{align*}
 Using the bounds
  \[ \rho_N \in L^{\infty}(\left[0,T \right[; H^{2s+1}(\T^d)) \ \ \ \text{and} \ \ \  \frac{1}{\rho_N} \in L^{\infty}(\left[0,T \right[;L^{\alpha}(\T^d))\]
  together with Lemma \ref{inverse_lemma}, we get that there exists a constant $C( E_{\reg}(\rho_0,u_0), \eta_1, \eta_2, \theta, T)$ such that
  \begin{equation} \label{bound_below_rho_N}
      \rho\geq C( E_{\reg}(\rho_0,u_0), \eta_1, \eta_2, \theta, T)>0,  
  \end{equation}  
  and we may set $u=V/\sqrt{\rho}$. \\
  
  \textbf{Step 3: The limit $N \rightarrow \infty$.} In order to pass to the limit $N \rightarrow \infty$ into equations \eqref{continuity_regularized_NSLK_eq} and \eqref{fluid_regularized_NSLK_eq}, we need the convergence of our functions $\rho_N$ and $u_N$ in a stronger sense. As $(\rho_N,u_N)$ satisfies the energy inequality \eqref{eq_energy_regularized_NSLK}, it also gets the uniform bounds of Remark \ref{regularity_regularized_NSLK}, in particular:
    \begin{align*}
  \rho_N \in L^{\infty} (\left[ 0,T \right[; H^{2s+1}(\T^d)) \cap L^2(\left[ 0,T \right[;  H^{2s+2}(\T^d)) , \\
  1/\rho_N \in L^{\infty} (\left[ 0,T \right[ \times \T^d), \ \ \ u_N \in L^2(\left[ 0,T \right[;  H^2(\T^d)),    
   \end{align*}
   and from the continuity equation \eqref{continuity_regularized_NSLK} satisfied by $\rho_N$ we know that $\partial_t \rho_N$ is bounded in $L^2( \left[ 0 ,T \right[; H^1(\T^d))$. By classical weak-convergence results and Ascoli-Arzelà type argument we get that
   \begin{align*}
 \rho_N \rightarrow \rho \ \text{strongly in} \ \mathcal{C}^0(\left[0,T \right[; H^{2s}(\T^d)), \\ 
 \rho_N \rightharpoonup \rho \ \text{weakly in} \ L^2(\left[0,T \right[;H^{2s+2}(\T^d)), \\
 \rho_N \rightharpoonup \rho \ \text{weakly in} \ H^1(\left[0,T \right[;H^1(\T^d)),
  \end{align*}
   and from the bound from below on $\rho_N$ \eqref{bound_below_rho_N} we also get that
   \[  1/\rho_N \rightarrow 1/\rho \ \text{strongly in} \ \mathcal{C}^0(\left[0,T \right[ \times \T^d ).   \]
   Furthermore, given the uniform bounds for $\rho_N$ and $u_N$, and since $(e_k)_k$ is orthogonal for the $H^2$- scalar product, we can get that $(\mathbb{P}_N\left[ \rho_N u_N \right])_N$ and $(\rho_N u_N)_N$ both converge in $L^2(\left[0,T \right[; H^1(\T^d))$ (see \cite{carles2019} for the details). Together with the fact that $(1/\rho_N)_N$ is uniformly bounded, we get that
   \[  u_N \rightarrow u \ \text{strongly in} \ L^2(\left[0,T \right[;H^{1}(\T^d)).  \]
   Note that the uniform estimates satisfied by $(\rho_N,u_N)$ also entail that $u_N$ (and so $u$) is bounded in $L^{\infty} (\left[0,T \right[; L^2(\T^d)) \cap L^2 (\left[0,T \right[; H^2(\T^d))$. The only remaining problematic term in \eqref{fluid_regularized_NSLK_eq} is $r_1 \rho_N |u_N|^2 u_N$. However as $u_N$ is bounded in $L^{\infty}(\left[ 0,T \right[; L^2(\T^d))$ and in $L^2(\left[ 0,T \right[; H^1(\T^d))$, we get by interpolation that $u_N$ is also bounded in $L^4(\left[ 0,T \right[; L^3(\T^d))$. \\
As each term in \eqref{continuity_regularized_NSLK_eq} and \eqref{fluid_regularized_NSLK_eq} is now handle by the previous regularities of $\rho_N$ and $u_N$, we can pass to the limit $N \rightarrow \infty$ into these equations, so finally $(\rho,u)$ is a weak solution of the regularized NSLK system \eqref{regularized_NSLK} in $\left[ 0,T \right[ \times \T^d$ with initial data $(\rho_0,u_0)$. 
 \end{proof}
 
In order to pass into the limits $\delta_1, \delta_2, \eta_1, \eta_2 \rightarrow 0$, we will need further estimates on our solution $(\rho,u)$. A common way to get other estimates is to introduce the following BD-entropy:
\begin{multline}
    \mathcal{E}_{\reg}(\rho,u)=\frac{1}{2} \int_{\T^d} \left(  \rho |u+\nu \nabla \log \rho|^2  +\hbar^2 |\nabla \sqrt{\rho} |^2 \right) + \lambda' \int_{\T^d} H(\rho) \\ +\frac{\eta_1}{\alpha + 1} \int_{\T^d} \rho^{-\alpha} + \frac{\eta_2}{2} \int_{\T^d} | \nabla \Delta^s \rho |^2 - r_0 \int_{\T^d} \log \rho,
\label{energy_BD_entropy_NSLK_reg} 
 \end{multline}
\begin{multline}
    \mathcal{D}_{\reg}(\rho,u) = \mu \int_{\T^d} \rho |u|^2  + \nu \int_{\T^d} \rho | \mathbb{A} u |^2 + \left( \nu \hbar^2 + \delta_1 \nu^2 + \frac{\delta_1 \hbar^2}{2}  \right)\int_{\T^d} \rho | \nabla^2 \log \rho |^2  \\
    + 4 (\nu + \delta_1) \int_{\T^d}|\nabla \sqrt{\rho} |^2     + \left( \frac{\eta_1 \nu \alpha}{4} + \frac{4 \delta_1 \eta_1}{10} \right) \int_{\T^d} | \nabla \rho^{-\frac{\alpha}{2}}|^2 \\
    + \eta_2 (\nu + \delta_1) \int_{\T^d} | \Delta^{s+1} \rho |^2 + \delta_2 \int_{\T^d} |\Delta u|^2 + r_0 \int_{\T^d} |u|^2 + r_1 \int_{\T^d} \rho|u|^4. 
\label{dissipation_BD_entropy_NSLK_reg}
\end{multline}

\begin{proposition}
 Let $T>0$ and assume $(\rho_0,u_0)$ satisfies the regularity \eqref{CI_regularized}. We denote $(\rho,u)$ the weak solution of \eqref{regularized_NSLK} constructed in Proposition \ref{regularized_solution_prop}. Then there exists constants $C_1$ and $C_2$ with dependencies mentioned in parentheses, such that:
 \begin{equation} \label{BD_entropy_regularized}
     \sup_{t \in \left[0 , T \right[} \mathcal{E}_{\reg}(\rho,u)(t) + \int_0^T \mathcal{D}_{\reg}(\rho,u)(s) ds \leq C_1(E_{\reg}(0)) + (\delta_1 + \delta_2) C_2(r_0,r_1,\eta_1,\eta_2,E_{\reg}(0)) .
 \end{equation}
\end{proposition}
\begin{proof}
The proof of this identity is mostly technical, and we refer to \cite{carles2019} for the details (as the new dissipation term $\mu \rho u$ is essentially harmless in the calculation). The idea is to differentiate and multiply by $\nu^2$ the continuity equation \eqref{continuity_regularized_NSLK}, which gives
\[ \nu^2( \partial_t( \rho \nabla \log \rho ) + \Div ( \rho \nabla \log \rho \otimes u) + \Div( \rho \nabla^\top u) - \delta_1 \Delta \nabla \rho) =0,  \]
and we also take $ \zeta =(\nu \nabla \log \rho) \xi$, where $\xi \in \mathcal{C}_0^{\infty}(\left[ 0,T \right[ \times \T^d; \T^d)$, as a test function in the weak formulation of the momentum equation \eqref{fluid_regularized_NSLK_eq}. We then combine these two equations and integrate in time. The technical part is then to prove that each term appearing in this equation is well defined, and using some classical inequality from functional analysis (namely H\"older or Young inequalities and Sobolev embedding) we finally get \eqref{BD_entropy_regularized}.
\end{proof}

\begin{remark} \label{remark_log_rho}
 Note that the presence of the term $-r_0 \int_{\T^d} \log \rho$ in $\mathcal{E}_{\reg}(\rho,u)$ prevents this quantity from being positive, however as for the logarithmic energy term we can get a lower bound for this term by controlling his negative part
 \[ -r_0 \int_{\rho >1} \log \rho \geq -r_0  \int_{\rho >1}  \rho \geq -r_0  \int_{\T^d}  \rho =  -r_0  \| \rho_0 \|_{L^1(\T^d)} \]
 as $\rho$ satisfies the mass conservation \eqref{conservation_mass}. Hence we can properly get some new uniform bounds on $(\rho,u)$ using equation \eqref{BD_entropy_regularized}.
\end{remark}

\subsection{NSLK with drag forces}
We are now going to prove the existence of a weak solution to the following system, which will be called NSLK system with drag forces:
\begin{subequations} \label{NSLK_drag}
\begin{empheq}[left=\empheqlbrace]{align}
& \partial_t \rho + \Div (\rho u) =0, \label{continuityNSLK_drag}  \\
& \partial_t (\rho u) + \Div \left( \rho u \otimes u \right) + \lambda \nabla \rho + \mu \rho u + r_0 u + r_1 \rho |u|^2u = \frac{\hbar^2 }{2} \rho \nabla \left( \frac{\Delta \sqrt{\rho}}{\sqrt{\rho}} \right) + \nu \Div \left( \rho \mathbb{D} u \right). \label{fluidNSLK_drag} 
\end{empheq} 
\end{subequations} 
We define the energy of this system and its corresponding dissipation from the ones of the previous regularized system,
\[ E_{\drag}(\rho,u):=  E_{\reg}(\rho,u)|_{\delta_1,\delta_2,\eta_1,\eta_2=0} \ \ \ \text{and} \ \ \  D_{\drag}(\rho,u):=  D_{\reg}(\rho,u)|_{\delta_1,\delta_2,\eta_1,\eta_2=0}, \]
as well as the BD-entropy and its corresponding flux,
\[ \mathcal{E}_{\drag}(\rho,u):=   \mathcal{E}_{\reg}(\rho,u)|_{\delta_1,\delta_2,\eta_1,\eta_2=0} \ \ \ \text{and} \ \ \  \mathcal{D}_{\drag}(\rho,u):=  \mathcal{D}_{\reg}(\rho,u)|_{\delta_1,\delta_2,\eta_1,\eta_2=0} .\]

 \begin{definition} 
Let $T >0$. We say that $(\rho,u)$ is a \textbf{weak solution} of the NSLK system with drag forces \eqref{NSLK_drag} in $\left[ 0,T \right[ \times \T^d$ with initial data $(\rho_0,u_0) \in L^1(\T^d) \times  L^2(\T^d)$, if the following holds:

\begin{enumerate}[label=(\roman*)]

    \item The global regularity:
    \[ \sqrt{\rho} \in \mathcal{C}^0(\left[0,T\right[; H^1(\T^d)), \ \ \ \nabla^2 \sqrt{\rho} \in L^2(\left[0,T\right[; L^2(\T^d)), \]
    \[  u \in L^2(\left[0,T\right[; L^2(\T^d)), \ \ \ \sqrt{\rho} u \in \mathcal{C}^0(\left[0,T\right[; L^2(\T^d)). \]
      
      \item For any test function $\eta \in \mathcal{C}_0^{\infty}(\left[ 0,T \right[ \times \T^d)$, equation \eqref{continuity_regularized_NSLK_eq} holds, and for any test function $\zeta \in \mathcal{C}_0^{\infty}(\left[ 0,T \right[ \times \T^d; \T^d)$, equation \eqref{fluid_regularized_NSLK_eq} holds, taking $\delta_1,\delta_2,\eta_1,\eta_2=0$ for both equations.
         
\end{enumerate}
\label{weak_drag_NSLK}
\end{definition}

\begin{remark}
As in \cite{carles2019}, remark that in presence of drag forces $r_0, r_1 >0$, $u$ is well defined as function, $\nabla u$ as a distribution and $\sqrt{\rho} \mathbb{D} u$ is also well defined, unlike in the original system \eqref{NSLK} without drag forces where the regularity induced by the energy estimate \eqref{eq_energy_NSLK} is insufficient to define $u$ and so $\sqrt{\rho} \mathbb{D} u$ has to be understood as $\mathbb{S}_N(u)$.
\end{remark}

\begin{proposition} \label{prop_NSLK_drag}
 Let $T>0$ and assume $(\rho_0,u_0)$ satisfies the regularity \eqref{CI_regularized} such that
 \[ E_{\drag}(\rho,u)(0)=E_{\drag}(\rho_0,u_0) < \infty \ \ \ \text{and} \ \ \  \mathcal{E}_{\drag}(\rho,u)(0)=\mathcal{E}_{\drag}(\rho_0,u_0) < \infty.  \]
 Then there exists a weak solution of the NSLK system with drag forces \eqref{regularized_NSLK} in $\left[ 0,T \right[ \times \T^d$ with initial data $(\rho_0,u_0)$. Furthermore, there exists constants $C_1$ and $C_2$ with dependencies mentioned in parentheses, such that
  \begin{equation} \label{BD_energy_drag}
     \sup_{t \in \left[0 , T \right[} E_{\drag}(\rho,u)(t) + \int_0^T D_{\drag}(\rho,u)(s) ds \leq C_1(E_{\drag}(0)) ,
\end{equation}
and
 \begin{equation} \label{BD_entropy_drag}
     \sup_{t \in \left[0 , T \right[} \mathcal{E}_{\drag}(\rho,u)(t) + \int_0^T \mathcal{D}_{\drag}(\rho,u)(s) ds \leq  C_2(E_{\drag}(0), \mathcal{E}_{\drag}(0)) .
 \end{equation}
\end{proposition}
\begin{proof}
The proof is the exact same as the one appearing in \cite{carles2019}, where the authors first let $\delta_1, \delta_2 \rightarrow 0$ and then perform the limit $\eta_1,\eta_2 \rightarrow 0$, so we refer to it for the details.
\end{proof}

\subsection{The limit $r_0,r_1 \rightarrow 0$}
In \cite{vasseur2018}, in order to pass to the limit $r_0,r_1 \rightarrow 0$, the authors have to introduce a new type of solutions to the Navier-Stokes-Korteweg system called renormalised solutions. In our framework these solutions are defined as follows:

 \begin{definition} 
Let $T >0$. We say that $(\rho,u)$ is a \textbf{renormalised weak solution} of the NSLK system with drag forces \eqref{regularized_NSLK} in $\left[ 0,T \right[ \times \T^d$ with initial data $(\sqrt{\rho_0},(\sqrt{\rho} u)_0) \in H^1(\T^d) \times  L^2(\T^d)$, if there exists locally integrable functions $\sqrt{\rho}$, $ \sqrt{\rho} u$ such that, by defining $\rho := \sqrt{\rho}^2$ and $u:= \sqrt{\rho} u/\sqrt{\rho}$, the following holds:

\begin{enumerate}[label=(\roman*)]

    \item The global regularity:
    \begin{align*}
    \sqrt{\rho} \in L^{\infty}(\left[ 0,T \right[;H^1(\T^d)), \ \ \ 
      \sqrt{\rho} u  \in L^{\infty}(\left[ 0,T \right[;L^2(\T^d)), \\
      \hbar \nabla^2 \sqrt{\rho} \in L^2(\left[ 0,T \right[;L^2(\T^d)), \ \ \ 
      \mathbb{T}_N(u) \in L^2(\left[ 0,T \right[;L^2(\T^d)),     \\
      \sqrt{\hbar} \nabla \rho^{1/4} \in L^4(\left[ 0,T \right[;L^4(\T^d)),  \ \ \
      r_1^{1/4} \rho^{1/4} u \in L^4(\left[ 0,T \right[;L^4(\T^d)), \\
      r_0^{1/2} u \in L^2(\left[ 0,T \right[;L^2(\T^d)),  \ \ \ 
      r_0 \log \rho \in L^{\infty}(\left[ 0,T \right[;H^1(\T^d)),
    \end{align*}  
      where $\mathbb{T}_N(u)$ is defined as in Definition \ref{weak_NSLK}, and with the compatibility condition
      \[\sqrt{\rho} \geq 0 \text{ a.e. on } \left[0,T\right[ \times \T^d, \ \ \ 
      \sqrt{\rho} u =0 \text{ a.e. on } \left\{ \rho=0 \right\}. \]
      
      \item For any function $\varphi \in W^{2,\infty}(\R^d)$, there exists two measures $f_{\varphi}$, $g_{\varphi} \in \mathcal{M}( \left]0, T \right[ \times \T^d)$ with
      \[ \| f_{\varphi} \|_{\mathcal{M}( \left]0, T \right[ \times \T^d)} + \| g_{\varphi}\|_{\mathcal{M}( \left]0, T \right[ \times \T^d)} \leq C \| \nabla^2 \varphi \|_{L^{\infty}(\R^d)},   \]
      where the constant $C$ depends only on the solution $(\rho,u)$ such that for any test function $\eta \in \mathcal{C}_0^{\infty}(\left[ 0,T \right[ \times \T^d)$, 
      \begin{equation} \label{continuity_renormalized}
          \int_0^T \int_{\T^d} (\rho \partial_t \eta + \rho u \cdot \nabla \eta )dx dt + \int_{\T^d} \rho_0 \eta(0) dx =0,
      \end{equation}  
      and for any test function $\zeta \in \mathcal{C}_0^{\infty}(\left[ 0,T \right[ \times \T^d; \T^d)$,
      \begin{multline} \label{fluid_renormalized}
      \int_0^T \int_{\T^d} \left( \rho \varphi(u) \cdot \partial_t \zeta + \rho \varphi(u)  u : \nabla \zeta + \left[ \lambda \rho \Div (\zeta) - \mu \rho u \cdot \zeta - r_0 u - r_1 \rho |u|^2 u
      \right. \right. \\ \left. \left.
      + \hbar^2  \nabla \sqrt{\rho} \otimes \nabla \sqrt{\rho} : \nabla \zeta - \frac{\hbar^2 }{4} \rho \Delta \Div \zeta  - \nu \sqrt{\rho} \mathbb{S}_N(u) \cdot \nabla \zeta \right] \varphi'(u) \right) dx dt  
      +  \int_{\T^d} \rho_0 u_0 \zeta(0) dx = \langle f_{\varphi}, \zeta \rangle, \end{multline}
      
      with $ \mathbb{S}_N(u)= (\mathbb{T}_N(u) + \mathbb{T}_N(u)^\top)/2$ and where $\mathbb{T}_N(u) $ is defined through the compatibility condition
      \[ \sqrt{\rho} \ \varphi_i'(u) (\mathbb{T}_N(u))_{j,k} = \partial_j(\rho \varphi_i'(u) u_k) - 2  \sqrt{\rho} u_k \partial_j \sqrt{\rho} + g_{\varphi}, \ \ \ \forall i,j,k \in \left\{ 1, \ldots, d \right\} .\]    
      
      \item  For any test function $\xi \in \mathcal{C}_0^{\infty}(\T^d)$, 
      \[ \lim\limits_{\substack{t \to 0 }} \int_{\T^d} \rho(t,x) \xi(x) dx =  \int_{\T^d} \rho_0(t,x) \xi(x) dx,\]
      \[ \lim\limits_{\substack{t \to 0 }} \int_{\T^d} \rho(t,x) u(t,x) \xi(x) dx =  \int_{\T^d} \rho_0(t,x) u_0(t,x) \xi(x) dx.\]

\end{enumerate}
\label{renormalised_NSLK}
\end{definition}

Note that renormalized solutions allow us to obtain some stability on our weak solutions, since the notion avoid the problem of concentration. In fact, in \cite{vasseur2018}, the authors have proved the following lemma:
\begin{lemma} \label{renormalized_lemma}
Let $T>0$, then:
\begin{itemize}
    \item For $r_0$, $r_1 \geq 0$, any renormalised weak solution of the NSLK system with drag forces \eqref{NSLK_drag} is also a weak solution (in the sense of Definition \ref{weak_drag_NSLK} if $r_0$, $r_1 >0$ or Definition \ref{weak_NSLK} is $r_0$, $r_1=0$). 
    \item For $r_0$, $r_1 >0$, the two notions are equivalent: any weak solution of the NSLK system with drag forces \eqref{NSLK_drag} is also a renormalized solution of the same system.
\end{itemize}
\end{lemma}

Thanks to Lemma \ref{renormalized_lemma} and Proposition \ref{prop_NSLK_drag}, we have constructed a renormalized solution of the NSLK system with drag forces \eqref{NSLK_drag}. In order to prove Theorem \ref{prop_NSLK}, the only remaining step is to prove the compactness of this renormalized solution in terms of the parameters $r_0$ and $r_1$.\\

\begin{proofof}{Theorem}{\ref{prop_NSLK}}
In this proof, we will denote $r=(r_0,r_1)$ and $(\rho_r,u_r)_r$ the sequence of weak solutions to the NSLK system with drag forces \eqref{NSLK_drag} constructed in Proposition \ref{prop_NSLK_drag}. Again, all the convergences below are made up to extraction that we do not relabel for conciseness.\\

We are first going to pass into the limit in the continuity equation \eqref{continuity_renormalized} only using a priori estimates (i) from Definition \ref{renormalised_NSLK} which do not depend on $r$. From the continuous embedding $H^1(\T^d) \xhookrightarrow{} L^6(\T^d)$ as $d \leq 3$, we get that $\sqrt{\rho_r} \in L^{\infty}( \left[0,T \right[; L^6(\T^d))$, so combined with the fact that $\sqrt{\rho_r} u_r \in L^{\infty}( \left[0,T \right[; L^2(\T^d))$ we get by Lemma \ref{lemma_power} that $\rho_r u_r$ is uniformly bounded in $L^{\infty}( \left[0,T \right[; L^{3/2}(\T^d))$. Moreover, using the continuity equation \eqref{continuity_renormalized} we get that $\partial_t \rho_r$ is uniformly bounded in $L^{\infty}( \left[0,T \right[; W^{-1,3/2}(\T^d))$. Writing $ \nabla \rho_r = 2 \sqrt{\rho_r} \nabla \sqrt{\rho_r}$, and recalling that $\nabla \sqrt{\rho_r} \in L^{\infty}( \left[0,T \right[; L^2(\T^d))$, we get that $\nabla \rho_r$ is also uniformly bounded in $L^{\infty}( \left[0,T \right[; L^{3/2}(\T^d))$, hence by Aubin-Lions lemma we get that
\[ \rho_r \rightarrow \rho \ \text{strongly in} \ \mathcal{C}^0(\left[0,T \right[; L^p(\T^d)), \ \ \ \text{for} \ 1 \leq p < 3. \]
From \eqref{fluid_renormalized} we get that $\partial(\rho_r u_r)$ is uniformly bounded in $L^2( \left[0,T \right[; H^{-N}(\T^d))$ for a $N$ large enough, and from the identity
\[ \nabla ( \rho_r u_r) = \sqrt{\rho_r} \sqrt{\rho_r} \nabla u_r + \nabla \rho_r \cdot u_r =  \sqrt{\rho_r} \mathbb{T}_N(u) + 2 \nabla  \sqrt{\rho_r} \cdot  \sqrt{\rho_r} u_r   \]
and the a priori estimates (i) from Definition \ref{renormalised_NSLK} we get that $\nabla ( \rho_r u_r)$ is uniformly bounded in  $L^2( \left[0,T \right[; L^{3/2}(\T^d))$. As we already know that $\rho_r u_r$ is uniformly bounded in $L^{\infty}( \left[0,T \right[; L^{3/2}(\T^d))$, we get from Aubin-Lions lemma that
\begin{align*}
    & \rho_r u_r \rightharpoonup \rho u \ \text{weakly in} \ \mathcal{C}^0(\left[0,T \right[; L^{3/2}(\T^d)), \ \ \ \text{for} \ 1 \leq p < 3. \\
    & \rho_r u_r\rightarrow \rho u \ \text{strongly in} \ L^p(\left[0,T \right[; L^q(\T^d)), \ \ \ \text{for} \ 1 \leq p < \infty \ \text{and} \ 1 \leq q < \frac{3}{2}. 
\end{align*} 
In particular, $(\rho_r)_r$ and $(\rho_r u_r)_r$ are uniformly bounded in $\mathcal{C}^0(\left[0,T \right[; L^1(\T^d))$, so we can pass to the limit $r \rightarrow 0$ in the continuity equation \eqref{continuity_renormalized}, and we also get part (iii) of Definition \ref{renormalised_NSLK}. \\

We are now going to pass to pass to the limit in the momentum equation \eqref{fluid_renormalized}. Using the previous convergences and the estimates independent of $r$ on $\nabla^2 \sqrt{\rho_r}$ and $\mathbb{T}_{N,r}$ that ensures some weak convergence of these quantities in $L^2(\left[0,T \right[; L^2(\T^d))$, we can pass to the limit $r \rightarrow 0$ in the left hand side of equation \eqref{fluid_renormalized}. Indeed, as in \cite{vasseur2018} and \cite{carles2019}, introducing $u=\rho u/\rho \mathbf{1}_{\left\{ \rho > 0 \right\}}$ we can show with the previous convergences that $\rho_r \rightarrow \rho$ and $u_r \rightarrow u$ a.e., and consequently that $\rho_r^{\alpha} \phi(u_r) \rightarrow \rho^{\alpha} \phi(u)$ in $L^p(\left[0,T \right[ \times \T^d)$ for any bounded $\phi : \R^d \rightarrow \R^d$, with $\alpha < 6$ and $p < 6/\alpha$. For the right hand side of equation \eqref{fluid_renormalized}, we remark that the sequence $(f_{\phi,r})_r$ is uniformly bounded in measures, so it converges to a measure $f_{\phi}$ with the same bound. Note that we can similarly pass to the limit in the renormalized compatibility condition for $\mathbb{T}_{N,r}$ and obtain the renormalized condition for $\mathbb{T}_N(u)$. \\

Hence $(\rho,u)$ defines a renormalized weak solution of the NSLK system \eqref{NSLK} in the sense of Definition \ref{renormalised_NSLK} taking $r_0=r_1=0$, which also stands as a weak solution of the NSLK system \eqref{NSLK}. By Fatou's lemma, we then get that $(\rho,u)$ satisfies \eqref{eq_energy_NSLK}, which completes the proof of Theorem \ref{prop_NSLK}.
\end{proofof}

\begin{remark}
Defining
\[ \mathcal{E}(\rho,u):=   \mathcal{E}_{\drag}(\rho,u)|_{r_0,r_1=0} \ \ \ \text{and} \ \ \  \mathcal{D}(\rho,u):=  \mathcal{D}_{\drag}(\rho,u)|_{r_0,r_1=0},\]
and as
\[  -r_0 \int_{\rho>1} \log \rho\geq -r_0 \| \rho_0 \|_{L^1(\T^d)} \ \ \ \text{and} \ \ \ 0 \leq -r_0 \int_{\rho<1} \log \rho \leq C_2(E_{\drag}(0), \mathcal{E}_{\drag}(0)) + r_0 \| \rho_0 \|_{L^1(\T^d)},  \]
from respectively Remark \ref{remark_log_rho} and equation \eqref{BD_entropy_drag}, which implies that
\[ -r_0 \int_{\T^d} \log \rho \underset{r_0 \to 0}{\longrightarrow} 0,  \]
we also gets from the previous proof that $(\rho,u)$ satisfies the following BD-entropy estimate (which will be useful in the following section):
 \begin{equation} \label{BD_entropy_NSLK}
     \sup_{t \in \left[0 , T \right[} \mathcal{E}(\rho,u)(t) + \int_0^T \mathcal{D}(\rho,u)(s) ds \leq  C(E(0), \mathcal{E}(0)) .
 \end{equation}

\end{remark}

\section{The isothermal Euler-Langevin-Korteweg system}
Following \cite{Bresch2019}, by denoting $w=u + \frac{\nu}{2} \nabla \log \rho$ and $v=  \log \rho$, we consider the augmented Navier-Stokes-Langevin-Korteweg system of equations:
\begin{subequations} \label{NSLK_augmented}
\begin{empheq}[left=\empheqlbrace]{align}
& \partial_t \rho + \Div (\rho u) =0, \label{continuityNSLK_augmented}  \\
& \partial_t (\rho w) + \Div \left( \rho w \otimes u \right) + \lambda' \nabla \rho + \mu \rho w= \frac{\hbar_{\nu}}{2} \Div( \rho \nabla v) + \frac{\nu}{2} \Div ( \rho \nabla w ), \label{fluidNSLK_augmented} \\
& \partial_t (\rho v) + \Div \left( \rho v \otimes u \right) + \Div \left( \rho \nabla u^\top \right) =0. \label{fluid2NSLK_augmented} 
\end{empheq} 
\end{subequations}
where we denote $\hbar_{\nu}^2 = \hbar^2 - \nu^2 >0 $ and $\lambda'= \lambda-\frac{\mu \nu}{2}>0$ (as we are letting $\nu \rightarrow 0$, cf Remark \ref{remark_lambda_pos} below).  
\begin{remark} \label{remark_lambda_pos}
Note that here the isothermal pressure term $\lambda \nabla \rho$ can absorb the contribution of the dissipative Langevin term, through the identity
\[ \lambda \nabla \rho + \mu \rho u = \lambda \nabla \rho + \mu \rho (w - \frac{\nu}{2} \nabla \log \rho) =  (\lambda - \frac{\mu \nu}{2} ) \nabla \rho + \mu \rho w  .\]
As we assume that $\lambda >0$ and as we are going to let $\nu \rightarrow 0$ in this section, we can take $\nu$ small enough such that $\lambda'=\lambda - \frac{\mu \nu}{2} >0$, which will be crucial in the following. In fact the compatibility of these two terms is a special feature of the isothermal pressure and the Langevin potential, which may not work for other pressure laws (for example the classical barotropic pressure $\lambda \rho^{\gamma}$, $\gamma> 1$), and reinforce the link between these two quantities.
\end{remark}

By denoting $\overline{v} =\hbar_{\nu} v/2$, we also rewrite the associated BD entropy estimate \eqref{eq_BD_entropy_NSLK} in terms of $\rho$, $w$ and $\overline{v}$, which stands as the energy estimate of the augmented system \eqref{NSLK_augmented}: 
\begin{equation*}  \mathcal{E}_{NSLK}(t) +  \int_0^t \mathcal{D}_{NSLK}(s) ds \leq E_0,  \end{equation*}
where
\begin{gather*}
    \mathcal{E}_{NSLK}(\rho,w,\overline{v})= \int_{\T^d} \left( \frac{1}{2} \rho |w|^2 +\frac{1}{2} \rho |\overline{v}|^2+ \lambda' H(\rho) \right) , \\
    \mathcal{D}_{NSLK}(\rho,w,\overline{v}) = \frac{\nu}{2} \int_{\T^d} \left(  \rho |\nabla w|^2 + \rho |\nabla \overline{v}|^2 + \frac{4\lambda'}{\hbar_{\nu}^2} \rho |\overline{v}|^2 \right)   + \mu \int_{\T^d} \rho |w|^2.
\end{gather*}

We also introduce the relative entropy entropy functional of the augmented NSLK system:
\begin{multline*}   \mathcal{E}_{NSLK}(\rho, w,\overline{v} | R,W,\overline{V})(t)
= \frac{1}{2} \int_{\T^d} \rho (|\overline{v}-\overline{V}|^2  +|w-W|^2 )+ \lambda' \int_{\T^d} H(\rho|R) \\ + \frac{\nu}{2} \int_0^t \int_{\T^d}  \rho \left( \left|\frac{ \T_N(\overline{v})}{\sqrt{\rho}}-\nabla \overline{V}\right|^2 + \left|\frac{ \T_N (w)}{\sqrt{\rho}}-\nabla W \right|^2  \right) + \mu \int_0^t \int_{\T^d}  \rho |w-W|^2,  
\end{multline*}
where $\T_N(\overline{v})$ and $\T_N (w)$ are defined through the compatibility conditions
\[ \sqrt{\rho} \T_N(\overline{v}) = \nabla( \sqrt{\rho} \sqrt{\rho} \overline{v} ) - 2  \sqrt{\rho} \overline{v} \otimes \nabla \sqrt{\rho} \ \ \ \text{and} \ \ \ \sqrt{\rho} \T_N(w) = \nabla( \sqrt{\rho} \sqrt{\rho} w ) - 2  \sqrt{\rho} w \otimes \nabla \sqrt{\rho}  .\]    

We now introduce the definitions of weak and strong solutions to the augmented NSLK system:

\begin{definition} 
Let $T >0$ and $(\rho_0,w_0,\overline{v}_0=\frac{\hbar}{2} \nabla \log \rho_0)$ such that $\mathcal{E}_{NSLK}(\rho_0,u_0,\overline{v}_0) < \infty$. We say that $(\rho,w,\overline{v})$ is a \textbf{weak solution} of the augmented system \eqref{NSLK_augmented} in $\left[ 0,T \right[ \times \T^d$ with initial data $(\rho_0,w_0,\overline{v}_0)$, if there exists locally integrable functions $\sqrt{\rho}$, $\sqrt{\rho}u $ such that, by defining $\rho := \sqrt{\rho}^2$, $\rho u:= \sqrt{\rho} \sqrt{\rho}u$, $v=  \log \rho$ and $w=u + \frac{\nu}{2} \nabla v$, the following holds:

\begin{enumerate}[label=(\roman*)]

    \item The global regularity (i) of Definition \ref{weak_NSLK} on $\sqrt{\rho}$ and $\sqrt{\rho}u$ is verified.
      
     \item For any test function $\eta \in \mathcal{C}_0^{\infty}(\left[ 0,T \right[ \times \T^d)$, $\rho$ and $\rho u$ verify equation \eqref{continuity_weak_NSLK}, for any test function $\zeta \in \mathcal{C}_0^{\infty}(\left[ 0,T \right[ \times \T^d; \T^d)$,
      \begin{multline*} \int_0^T \int_{\T^d} \left( \rho w \cdot \partial_t \zeta + \sqrt{\rho} w \otimes \sqrt{\rho} u : \nabla \zeta +\lambda' \rho \Div (\zeta) - \mu \rho w \cdot \zeta      +  \sqrt{ \rho} \T_N(\overline{v}) : \nabla \zeta \right. \\ \left. + \frac{\nu}{2} \sqrt{ \rho} \T_N(w) : \nabla \zeta   \right) dx dt  
      +  \int_{\T^d} \rho_0 u_0 \zeta(0) dx =0, \label{fluid_weak_augmented_NSLK} \end{multline*}
     and for any test function $\xi \in \mathcal{C}_0^{\infty}(\left[ 0,T \right[ \times \T^d; \T^d)$, 
     \begin{equation*}
         \int_0^T \int_{\T^d} \left( \rho v \cdot \partial_t \xi  \sqrt{\rho} v \otimes \sqrt{\rho} u : \nabla \xi + \sqrt{\rho} \T_N(u)^\top : \nabla \xi\right) dx dt  
      +  \int_{\T^d} \rho_0 v_0 \xi(0) dx =0.
     \end{equation*}
      
      \item  The compatibility condition (iii) of Definition \ref{weak_NSLK} on $\rho$ and $\rho u$ is satisfied.
\end{enumerate}
\label{weak_NSLK_augmented}
\end{definition}

\begin{remark}
Let us remark here that weak solutions of the NSLK system \eqref{NSLK} are also weak solutions of the augmented NSLK system \eqref{NSLK_augmented}. In fact, taking the gradient of equation \eqref{continuityNSLK} (which is satisfied in the distribution sense), we get that
\[ \partial_t \nabla \rho + \Div( \nabla (\rho u)^\top) =0,   \]
so by definition of $\sqrt{\rho} \T_N$ and expression of $v$ we can write
\[ \partial(\rho v) + \Div (\rho v \otimes u) + \Div(\sqrt{\rho} \T_N(u)^\top) =0.   \]
\end{remark}

\begin{definition} 
Let $T >0$ and $(R_0,W_0,\overline{V}_0=\frac{\hbar_{\nu}}{2} \nabla \log R_0)$ such that $\mathcal{E}_{NSLK}(R_0,W_0,\overline{V}_0) < \infty$. We say that $(R,W,\overline{V})$ is a \textbf{strong solution} of the augmented system \eqref{NSLK_augmented} in $\left[ 0,T \right[ \times \T^d$ with initial data $(R_0,W_0,\overline{V}_0)$, if the following holds:

\begin{enumerate}[label=(\roman*)]

    \item The global regularity:
    \begin{gather*} 
    0 < \inf_{\left] 0, T \right[ \times \T^d} R \leq R \leq \sup_{\left] 0, T \right[ \times \T^d} R < \infty, \\
    \nabla R \in L^2(\left[ 0,T \right[;L^{\infty}(\T^d)) \cap L^1(\left[ 0,T \right[;W^{1,\infty}(\T^d)) \\
    W, \overline{V} \in  L^{\infty}(\left[ 0,T \right[;W^{2,\infty}(\T^d)) \cap W^{1,\infty}(\left[ 0,T \right[;L^{\infty}(\T^d)) \\
    \partial_t H'(R), \ \nabla H'(R) \in L^1(\left[ 0,T \right[;L^{\infty}(\T^d)).
      \end{gather*}
      
    \item The function $R,\overline{V}$ and $W$ satisfies, for a.a. $(t,x) \in \left[ 0,T \right] \times \T^d$,
\begin{subequations} \label{NSLK_augmented_strong}
\begin{empheq}{align}
    &   \partial_t R + \Div (R U) =0, \label{continuityNSLK_augmented_strong} \\
&  R (\partial_t W  + \nabla W U) + \lambda' \nabla R + \mu R W = \frac{\hbar_{\nu}}{2} \Div( R \nabla \overline{V} ) + \frac{\nu}{2} \Div ( R \nabla W ), \label{fluidNSLK_augmented_strong} \\
&  R (\partial_t \overline{V}  + \nabla \overline{V} U) + \Div \left( R \nabla U^\top \right) =0, \label{fluid2NSLK_augmented_strong}  
\end{empheq} 
\end{subequations}
    where $U=W - \frac{\nu}{2} V$, and $\overline{V} = \frac{\hbar_{\nu}}{2} V$.
\end{enumerate}
\label{strong_sol_def_NSLK}
\end{definition}

\subsection{A Gronwall inequality}
We are going to mimic the steps of the proof given in \cite{Bresch2019}, giving a special attention to the isothermal pressure law $\lambda' \rho$ and to the new dissipation term:
\[ \mu \int_0^t \int_{\T^d}  \rho |w-W|^2. \]

\begin{proposition} \label{first_ineq}
Let $(\rho, w, \overline{v})$ be a weak solution of the augmented system \eqref{NSLK_augmented}, and let 
\[ R \in \mathcal{C}^1(\left[0 , T \right] \times \T^d ), \ \ \ R > 0, \ \ \ W,\overline{V} \in \mathcal{C}^2(\left[0,T\right] \times  \T^d ).\]
 Then we have the following inequality:
\begin{align*}  
\mathcal{E}_{NSLK}(\rho, w,\overline{v} | R,W,\overline{V})(t) & \leq  \mathcal{E}_{NSLK}(\rho, w,\overline{v} | R,W,\overline{V}) (0) \\
& + \int_0^t \int_{\T^d} \rho ( \partial_t \overline{V} \cdot (\overline{V} - \overline{v}) + (\nabla \overline{V} u) \cdot(\overline{V} - \overline{v})) \\
& + \int_0^t \int_{\T^d} \rho ( \partial_t W \cdot (W - w) + (\nabla W u) \cdot(W - w)) \\
& - \lambda' \int_0^t \int_{\T^d}  ( \partial_t(H'(R))(\rho - R) + \rho \nabla (H'(R)) \cdot u + \rho \Div W) \\
& - 2 \nu \lambda' \int_0^t \int_{\T^d} |\nabla \sqrt{\rho}|^2 +\mu \int_0^t \int_{\T^d} \left( \rho |W|^2 -\rho w \cdot W \right) + I_1,
\end{align*}
where
\begin{align*}   I_1 = \frac{\nu}{2} \int_0^t \int_{\T^d} \rho ( | \nabla \overline{V} | + | \nabla W|^2 ) - \sqrt{\rho} ( \T(\overline{v}) : \nabla \overline{V} + \T (w) : \nabla W) \\
+\frac{\hbar_{\nu}}{2} \int_0^t \int_{\T^d} \sqrt{\rho} ( \T(\overline{v}) :  \nabla W  - (\T (w))^\top : \nabla \overline{V}).  
\end{align*}
\end{proposition}
 
 \begin{proof}
 As we know that $(\rho, w, \overline{v})$ is a solution of \eqref{NSLK_augmented}, we can use \eqref{eq_energy_NSLK}, and using the fact that for any function $f \in W^{1,1}( \left[0,T \right[; L^1(\T^d))$ we have
 \[ \int_{\T^d} f(t,x) dx - \int_{\T^d} f(0,x) dx = \int_0^t \int_{\T^d} \partial_t f(s,x) dx ds, \]
 we get that
 \begin{align*}  
\mathcal{E}_{NSLK}(t) - \mathcal{E}_{NSLK}(0) & \leq  \int_0^t \int_{\T^d} \partial_t \left(  \frac{1}{2} \rho | \overline{V}|^2 - \rho \overline{v} \cdot \overline{V} \frac{1}{2} \rho | W|^2 - \rho w \cdot W \right) \\
& + \frac{\nu}{2} \int_0^t \int_{\T^d} \rho ( | \nabla \overline{V} | + | \nabla W|^2 ) - 2 \sqrt{\rho} ( \T(\overline{v}) : \nabla \overline{V} + \T (w) : \nabla W) \\
& - \lambda' \int_0^t \int_{\T^d} \partial_t \left( H(R) + H'(R)(\rho - R)  \right)- 2 \nu \lambda' \int_0^t \int_{\T^d} |\nabla \sqrt{\rho}|^2   \\
& + \mu \int_0^t \int_{\T^d} \left( \rho |W|^2 - 2 \rho w \cdot W \right) .
\end{align*}
We are now going to use equations from the augmented system \eqref{NSLK_augmented} in order to develop the terms in the previous inequality. Using equation \eqref{continuityNSLK_augmented}, we get that 
\begin{align*}
    \int_0^t \int_{\T^d} \partial_t \left( H(R) + H'(R)(\rho - R)  \right) 
& = \int_0^t \int_{\T^d} ( H'(R) \partial_t R + \partial_t( H'(R) ) (\rho - R) + H'(R) \partial_t ( \rho - R)) \\
& = \int_0^t \int_{\T^d} (  \partial_t( H'(R) ) (\rho - R)  - H'(R) \Div ( \rho u) ) \\
& = \int_0^t \int_{\T^d} (  \partial_t( H'(R) ) (\rho - R)  + \rho \nabla(H'(R)) \cdot u ).
\end{align*}
Then using respectively equations \eqref{fluidNSLK_augmented} and \eqref{fluid2NSLK_augmented}, we have
\begin{multline*}
    \partial_t( \rho w \cdot W) = \partial_t ( \rho w ) \cdot W + \rho w \cdot \partial_t W \\
    = \langle - \Div \left( \rho w \otimes u \right) - \lambda' \nabla \rho - \mu \rho w + \frac{\hbar_{\nu}}{2} \Div( \sqrt{\rho}  \T (\overline{v}) ) + \frac{\nu}{2} \Div ( \sqrt{\rho}  \T (w) ) | W \rangle_{W^{-2,1}(\T^d) \times W^{2, \infty} ( \T^d)} + \rho w \cdot \partial_t W,
\end{multline*}
and 
\begin{multline*}
    \partial_t( \rho \overline{v} \cdot \overline{V}) = \partial_t ( \rho \overline{v} ) \cdot \overline{V} + \rho \overline{v} \cdot \partial_t \overline{V} \\
  = \langle - \Div \left( \rho \overline{v} \otimes u \right) - \frac{\hbar_{\nu}}{2} \Div( \sqrt{\rho}  \T (w)^\top ) + \frac{\nu}{2} \Div ( \sqrt{\rho}  \T (\overline{v})   | \overline{V} \rangle_{W^{-2,1}(\T^d) \times W^{2, \infty} ( \T^d)} + \rho \overline{v} \cdot \partial_t \overline{V}.
 \end{multline*}
 
 Finally, we develop the quantities
 \[ \partial_t(  \frac{1}{2} \rho | \overline{V}|^2 ) = \frac{1}{2} \partial_t \rho |\overline{V}|^2 + \rho \overline{V} \cdot \partial_t \overline{V}, \ \ \  \partial_t(  \frac{1}{2} \rho | W|^2 ) = \frac{1}{2} \partial_t \rho |W|^2 + \rho W \cdot \partial_t W,   \]
 and since $\nabla \overline{v}$, $\nabla \overline{V}$ are symmetric matrices (recall that $v$ and $V$ are gradient of functions), we get from \eqref{continuityNSLK_augmented} and integration by parts:
 \begin{align*}  
\mathcal{E}_{NSLK}(t) - \mathcal{E}_{NSLK}(0) & \leq  \int_0^t \int_{\T^d} \rho  \partial_t \overline{V} \cdot (\overline{V} - \overline{v}) + \int_0^t \int_{\T^d} \rho (\nabla \overline{V} u) \cdot(\overline{V} - \overline{v}) \\
& + \int_0^t \int_{\T^d} \rho  \partial_t W \cdot (W - w) +  \int_0^t \int_{\T^d} \rho  (\nabla W u) \cdot(W - w) \\
& + \frac{\nu}{2} \int_0^t \int_{\T^d} \rho ( | \nabla \overline{V} | + | \nabla W|^2 ) - 2 \sqrt{\rho} ( \T(\overline{v}) : \nabla \overline{V} + \T (w) : \nabla W) \\
& - \lambda' \int_0^t \int_{\T^d}   \partial_t (H'(R))  (\rho - R)  - \lambda' \int_0^t \int_{\T^d} \rho \nabla(H'(R)) \cdot u  \\
& - 2 \nu \lambda' \int_0^t \int_{\T^d} |\nabla \sqrt{\rho}|^2 + \mu \int_0^t \int_{\T^d} \left( \rho |W|^2 - 2 \rho w \cdot W \right) \\
& - \lambda' \int_0^t \int_{\T^d} \rho \Div W + \mu \int_0^t \int_{\T^d} \rho w \cdot W   \\
& + \frac{\nu}{2} \int_0^t \int_{\T^d} \sqrt{\rho} ( \T(\overline{v}) : \nabla \overline{V} + \T (w) : \nabla W)  \\
& +\frac{\hbar_{\nu}}{2} \int_0^t \int_{\T^d} \sqrt{\rho} ( \T(\overline{v}) :  \nabla W  - (\T (w))^\top : \nabla \overline{V}) .
\end{align*}
 The result follows from simplifying the previous expression.
\end{proof}

\begin{proposition} \label{secund_ineq}
Let $(R, W,\overline{V})$ be a strong solution of \eqref{NSLK_augmented}, then any weak solution $(\rho, w, \overline{v})$ of the augmented system \eqref{NSLK_augmented} satisfies:
\begin{align*}  \label{secund_inequality}
\mathcal{E}_{NSLK}(t) - \mathcal{E}_{NSLK}(0) & \leq \int_0^t \int_{\T^d} \rho (\nabla \overline{V} (u-U)) \cdot(\overline{V} - \overline{v}) +  \int_0^t \int_{\T^d} \rho  (\nabla W (u-U) ) \cdot(W - w) \\
& + \lambda' \frac{\nu}{ \hbar_{\nu}} \int_0^t \int_{\T^d}  ( \rho \nabla (H'(R)) \cdot ( \overline{v} - \overline{V}) - \rho \Div \overline{V} ) - 2 \nu \lambda' \int_0^t \int_{\T^d} |\nabla \sqrt{\rho}|^2  \\
&  + \frac{\nu}{2}  \int_0^t \int_{\T^d}  \frac{\rho}{R} (  \Div ( R \nabla \overline{V} ) \cdot ( \overline{V} - \overline{v}) + \Div (R \nabla W ) \cdot ( W-w))\\
& +  \frac{\hbar_{\nu}}{2} \int_0^t \int_{\T^d}  \frac{\rho}{R} (  \Div ( R \nabla \overline{V} ) \cdot (W-w) - \Div (R \nabla W^\top ) \cdot ( \overline{V} - \overline{v})) + I_1.
\end{align*}
\end{proposition}
\begin{proof}
Multiplying equation \eqref{continuityNSLK_augmented_strong} by $\frac{\rho}{R} ( W-w)$ and equation \eqref{fluid2NSLK_augmented_strong} by $\frac{\rho}{R} ( \overline{V}-\overline{v})$, and integrating with respect to time and space, we can replace the terms $\rho ( \partial_t W \cdot (W - w))$ and $\rho ( \partial_t \overline{V} \cdot (\overline{V} - \overline{v}))$ in the inequality of Proposition \ref{first_ineq}, which gives:
\begin{align*}  
\mathcal{E}_{NSLK}(t) - \mathcal{E}_{NSLK}(0) & \leq \frac{\nu}{2} \int_0^t \int_{\T^d} \frac{\rho}{R} \Div ( R \nabla \overline{V}) \cdot (\overline{V}-\overline{v}) - \frac{\hbar_{\nu}}{2} \int_0^t \int_{\T^d} \frac{\rho}{R} \Div ( R \nabla W^\top) \cdot (\overline{V}-\overline{v}) \\
& + \int_0^t \int_{\T^d} \rho (\nabla \overline{V} (u-U)) \cdot(\overline{V} - \overline{v}) +  \int_0^t \int_{\T^d} \rho  (\nabla W (u-U) ) \cdot(W - w) \\
&  + \frac{\nu}{2}  \int_0^t \int_{\T^d} \frac{\rho}{R} \Div (R \nabla W ) \cdot ( W-w) + \frac{\hbar_{\nu}}{2} \int_0^t \int_{\T^d}  \frac{\rho}{R}   \Div ( R \nabla \overline{V} ) \cdot (W-w) \\
& - 2 \nu \lambda' \int_0^t \int_{\T^d} |\nabla \sqrt{\rho}|^2 + I_1 + \lambda' I_2,
\end{align*}
where
\begin{align*} 
   I_2= - \int_0^t \int_{\T^d} \frac{\rho}{R} \nabla R \cdot ( W-w) - \int_0^t \int_{\T^d}  (\partial_t(H'(R))(\rho - R) +\rho \nabla (H'(R)) \cdot u) \\
   + \int_0^t \int_{\T^d}  H'(R) \partial_t R +\int_0^t \int_{\T^d} ( R \Div U - \rho \Div W).
\end{align*} 
Note that the dissipation term in $\mu$ has disappeared from the expression. In fact, by multiplying by $\frac{\rho}{R} ( W-w)$, the contribution of \eqref{continuityNSLK_augmented_strong} in the previous inequality is
\[ - \mu \int_0^t \int_{\T^d} \frac{\rho}{R} ( W-w) \cdot R W = -\mu \int_0^t \int_{\T^d} \left( \rho |W|^2 -\rho w \cdot W \right), \]
which is exactly the opposite of the one in the expression of Proposition \ref{first_ineq}. Using equation \eqref{continuityNSLK_augmented_strong} $H'(R) \partial_t R + H'(R) \Div(rU)=0$, and as $R \nabla (H'(R))= \nabla R$, we get by integration by parts that
\[ 0 = \int_0^t \int_{\T^d} (H'(R) \partial_t R -R \nabla( H'(R)) \cdot U) = \int_0^t \int_{\T^d}  (H'(R) \partial_t R  + R \Div U).   \]
Observing that
\[ \partial_t(H'(R)) =  - \Div U - H''(R) \nabla R \cdot U = - \Div U - \nabla(H'(R)) \cdot U,  \]
we get by integration by parts that
\begin{align*}
    I_2 = \int_0^t \int_{\T^d} \rho \nabla ( H'(R)) \cdot ( - W+w+U-u) + \int_0^t \int_{\T^d} \left( \Div U (\rho - R) - R \nabla ( H'(R))\cdot U -\rho \Div W \right) \\
    = \lambda' \frac{\nu}{ \hbar_{\nu}} \int_0^t \int_{\T^d}  ( \rho \nabla (H'(R)) \cdot ( \overline{v} - \overline{V}) - \rho \Div \overline{V} ) ,
\end{align*}
recalling that $U=W - \frac{\nu}{2} V$ and $\overline{V} = \frac{\hbar_{\nu}}{2} V$, which gives the result.
\end{proof}

We can now state the main theorem of this section:
\begin{theorem} \label{gronwall_strong}
Let $(R, W,\overline{V})$ be a strong solution of \eqref{NSLK_augmented} and assume $\lambda' >0$, then any weak solution $(\rho, w, \overline{v})$ of the augmented system \eqref{NSLK_augmented} satisfies:
\begin{equation} \label{gronwall_ineq}
    \mathcal{E}_{NSLK}(t) - \mathcal{E}_{NSLK}(0) \leq C \left(1+ \frac{\nu}{\hbar_{\nu}} \right) \int_0^t \mathcal{E}_{NSLK}(s) ds,
\end{equation}
where $C=C(R,W,\overline{V})$ denotes a constant independent of $\lambda$, $\mu$, $\nu$ and $\hbar$.
\end{theorem}
\begin{proof}
The idea is to bound every term on the right hand side of the inequality of Proposition \ref{secund_ineq} by a multiple of $\int_0^t \mathcal{E}_{NSLK}$. Recalling that $ W, \overline{V} \in  L^{\infty}(\left[ 0,T \right[;W^{2,\infty}(\T^d)) \cap W^{1,\infty}(\left[ 0,T \right[;L^{\infty}(\T^d))$, we immediately get by Cauchy-Schwarz inequality that
\begin{align*}
    \int_0^t \int_{\T^d} \rho (\nabla \overline{V} (u-U)) \cdot(\overline{V} - \overline{v}) +  \int_0^t \int_{\T^d} \rho  (\nabla W (u-U) ) \cdot(W - w)  \\
   \leq C   \int_0^t \int_{\T^d} \rho (|u-U|^2+ |\overline{v}-\overline{V}|^2+  |w-W|^2).
\end{align*}

For the second term of the inequality of Proposition \ref{secund_ineq}, as $\nabla (H'(R)) = \nabla \log R = V$, we get by integration by parts that
\begin{multline*}
    \lambda' \frac{\nu}{ \hbar_{\nu}} \int_0^t \int_{\T^d}  ( \rho \nabla (H'(R)) \cdot ( \overline{v} - \overline{V}) - \rho \Div \overline{V} ) - \frac{\nu}{2} \lambda' \int_0^t \int_{\T^d} \frac{1}{\rho} |\nabla \rho|^2 \\
    = - \lambda'\frac{\nu}{\hbar_{\nu}} \int_0^t \int_{\T^d} \left( \rho \overline{V} \cdot (\overline{V} -  \overline{v} ) - \lambda'\rho \overline{V} \cdot \overline{v} \right) - \frac{\nu}{\hbar_{\nu}} \int_0^t \int_{\T^d} \rho |\overline{v}|^2 = - \lambda' \frac{\nu}{\hbar_{\nu}} \int_0^t \int_{\T^d} \rho |\overline{V}-\overline{v}|^2.    
\end{multline*}

For the third and fourth lines of the inequality of Proposition \ref{secund_ineq}, we can also show in the same way that
\begin{multline*}
    \frac{\nu}{2}  \int_0^t \int_{\T^d}  \frac{\rho}{R} (  \Div ( R \nabla \overline{V} ) \cdot ( \overline{V} - \overline{v}) + \Div (R \nabla W ) \cdot ( W-w))\\
 +  \frac{\hbar_{\nu}}{2} \int_0^t \int_{\T^d}  \frac{\rho}{R} (  \Div ( R \nabla \overline{V} ) \cdot (W-w) - \Div (R \nabla W^\top ) \cdot ( \overline{V} - \overline{v})) + I_1 \\
 = \frac{\nu}{2} \int_0^t \int_{\T^d} \rho (V-v) \cdot \left[ \nabla \overline{V} ( \overline{V} -\overline{v}) + \nabla W^\top (W-w)    \right] \\
 + \frac{\hbar_{\nu}}{2} \int_0^t \int_{\T^d} \rho (V-v) \cdot \left[ \nabla \overline{V}(W-w)  - \nabla W^  (\overline{V} -\overline{v})    \right],
\end{multline*}
hence we get the result by bounding these integrals like the first term.
\end{proof}

\subsection{The viscous limit $\nu \rightarrow 0$}
\begin{theorem}
Let $T>0$, $\lambda'>0$ and $(\rho_0,u_0,\overline{v}_0=\frac{\hbar}{2} \nabla \log \rho_0)$ such that $\mathcal{E}_{ELK}(\rho_0,u_0,\overline{v}_0) < \infty$. Let $(\rho^{\nu},u^{\nu},\overline{v}^{\nu}= \frac{\hbar_{\nu}}{2} \nabla \log \rho^{\nu})$ be a weak solution to the augmented NSLK system \eqref{NSLK_augmented} in $\left[0 , T \right[ \times \T^d$ with initial data $(\rho_0,  u_0,  \overline{v}_0^{\nu} )$. Let $(\rho,u,\overline{v})$ be the weak limit of $(\rho^{\nu},u^{\nu},\overline{v}^{\nu})$ when $\nu$ tends to 0 in the sense
\[ \rho^{\nu} \rightharpoonup \rho \ \text{weakly in} \  L^{\infty}(\left[ 0,T \right[; L^1(\T^d)), \]
\[ \sqrt{\rho^{\nu}} w^{\nu} \rightharpoonup \sqrt{\rho} u \ \text{weakly in} \  L^{\infty}(\left[ 0,T \right[; L^2(\T^d)), \]
\[ \sqrt{\rho^{\nu}} \overline{v}^{\nu} \rightharpoonup \sqrt{\rho} \overline{v} \ \text{weakly in} \  L^{\infty}(\left[ 0,T \right[; L^2(\T^d)), \]
with $\rho \overline{v} = \frac{\hbar}{2} \nabla \rho$. Then $(\rho,u,\overline{v})$ is a dissipative solution of the augmented ELK system \eqref{ELK_augmented} in $\left[0 , T \right[ \times \T^d$ with initial data $(\rho_0, u_0, \overline{v}_0)$.
\end{theorem}

\begin{proof}
Let $U$ be a smooth enough function such that, defining $(R,\overline{V}, \mathscr{E})$ through \eqref{ELK_augmented_def} and
\[  \overline{V}^{\nu}= \frac{\hbar}{2}\nabla \log R \ \ \ \text{and} \ \ \ W^{\nu}=U+\frac{\nu}{2} V,  \]
we have the global regularity (i) of Definition \ref{strong_sol_def_NSLK} on $(R, W^{\nu},\overline{V}^{\nu})$ and the compatibility condition $\int R_0 = \int \rho_0$. We also define
\[ \mathscr{E}^{\nu}(R,U)=R(\partial_t U + U \cdot \nabla U) + \lambda \nabla R +\mu \rho U- 2 \frac{\nu}{2} \Div(R \mathbb{D} U) - \frac{\hbar}{2}( \Div(R \nabla \overline{V})),  \]
such that
\[ \mathscr{E}^{\nu}(R,U) =  \mathscr{E}(R,U)  -  \nu \Div(R \mathbb{D} U). \]
Using equation \eqref{continuityELK_augmented_def}, we can easily check that $(R, W^{\nu},\overline{V}^{\nu})$ satisfies \eqref{fluid2NSLK_augmented_strong} and that
\[ \mathscr{E}^{\nu}(R,U) = R(\partial_t W + U \cdot \nabla W) + \lambda' \nabla R + \mu R W - \frac{\nu}{2} \Div(R \nabla W) - \frac{\hbar_{\nu}}{2} \Div(R \nabla \overline{V}).\]
As $\mathscr{E}^{\nu}(R,U)$ is not necessarily equal to 0, $(R, W^{\nu},\overline{V}^{\nu})$ fails to be a strong solution of the augmented system \eqref{NSLK_augmented} in the sense of Definition \ref{strong_sol_def_NSLK}, however we can show in the exact same way as for the proof of Theorem \ref{gronwall_strong} that
\begin{equation*} \label{gronwall_ineq_approx}
    \mathcal{E}_{NSLK}(t) - \mathcal{E}_{NSLK}(0) \leq C^{\nu} \int_0^t \mathcal{E}_{NSLK}(s) ds + b_{NSLK}^{\nu}(t),
\end{equation*}
where 
\[ b_{NSLK}^{\nu}(t) = \int_0^t \int_{\T^d} \frac{\rho^{\nu}}{R} ( \mathscr{E}^{\nu}(R,U) \cdot (W^{\nu}-w^{\nu})) \ \ \ \text{and} \ \ \ C^{\nu}=C \left(1+ \frac{\nu}{\hbar_{\nu}} \right) . \]
By Gronwall lemma, we get that
\begin{equation} \label{gronwall_end}
    \mathcal{E}_{NSLK}(t) \leq \mathcal{E}_{NSLK}(0) e^{C^{\nu}t} + C^{\nu} \int_0^t b_{NSLK}^{\nu}(s) e^{C^{\nu} (t-s)} ds +b_{NSLK}^{\nu}(t).
\end{equation}
By definition, we have
\[ \frac{1}{2} \int_{\T^d} \rho (|\overline{v}-\overline{V}|^2  +|w-W|^2 )+ \lambda' \int_{\T^d} H(\rho|R) + \mu \int_0^t \int_{\T^d}  \rho |w-W|^2  \leq \mathcal{E}_{NSLK}(\rho^{\nu}, w^{\nu},\overline{v}^{\nu} | R,W^{\nu},\overline{V}^{\nu})(t), \]
so \eqref{gronwall_end} gives
\begin{multline*}
\frac{1}{2} \int_{\T^d} \rho^{\nu} (|\overline{v}^{\nu}-\overline{V}^{\nu}|^2  +|w^{\nu}-W^{\nu}|^2 )+ \lambda' \int_{\T^d} H(\rho^{\nu}|R) + \mu \int_0^t \int_{\T^d}  \rho |w-W|^2 \\
\leq \mathcal{E}_{NSLK}(\rho^{\nu}, w^{\nu},\overline{v}^{\nu} | R,W^{\nu},\overline{V}^{\nu})(0) e^{C^{\nu}t} + C^{\nu} \int_0^t b_{NSLK}^{\nu}(s) e^{C^{\nu} (t-s)} ds +b_{NSLK}^{\nu}(t).   
\end{multline*}
We are now going to pass to the limit $\nu \rightarrow 0$ in the previous inequality. By the lower semi-continuity of the term $\mathcal{E}_{NSLK}(\rho^{\nu}, w^{\nu},\overline{v}^{\nu} | R,W^{\nu},\overline{V}^{\nu})$, the left hand-side is not smaller than
\[ \frac{1}{2} \int_{\T^d} \rho (|\overline{v}-\overline{V}|^2  +|u-U|^2 )+ \lambda \int_{\T^d} H(\rho|R) + \mu \int_0^t \int_{\T^d}  \rho |u-U|^2,    \]
which corresponds to $\mathcal{E}_{ELK}(\rho, u,\overline{v} | R,U,\overline{V})(t)$. On the right-hand side, as $b_{NSLK}^{\nu}(t)$ tends to
\[  b_{ELK}(t)(t) = \int_0^t \int_{\T^d} \frac{\rho}{R} \mathscr{E} \cdot (U-u),  \]
for all $t \geq 0$, and by the direct limit of $\mathcal{E}_{NSLK}(\rho^{\nu}, w^{\nu},\overline{v}^{\nu} | R,W^{\nu},\overline{V}^{\nu})(0)$, we conclude that
\[ \mathcal{E}_{ELK}(t) \leq\mathcal{E}_{ELK}(0) e^{Ct} + b_{ELK}(t) + C \int_0^t b_{ELK}(s) e^{C(t-s)} ds   ,\]
which shows that $(\rho,u,\overline{v})$ is indeed a dissipative solution in the sense of Definition \ref{dissipative_solution}.
\end{proof}

\appendix

\section{Definition of the operators and technical lemmas}
We recall here all the definitions of the operators used in this article. We denote $u,v$ two vectors and $\sigma = (\sigma_{ij})_{1\leq i,j \leq d}$, $\tau = (\tau_{ij})_{1\leq i,j \leq d}$ two tensor fields defined on $\Omega \subset \R^d$ smooth enough. First, denoting by $u_1, \ldots, u_d$ the coordinates of $u$, we respectively call \textit{divergence}, \textit{gradient} and \textit{laplacian} of $u$ the following functions:
\[\Div(u) = \sum_{i=1}^d \frac{\partial u_i}{\partial x_i}, \ \ \ \nabla u = \left( \frac{\partial u_i}{\partial x_j}\right)_{1 \leq i,j \leq d} \ \ \ \text{and} \ \ \   \Delta u = \Div ( \nabla u).    \]
We also call \textit{tensor product} of $u$ and $v$ the tensor given by:
\[ u \otimes v = (u_i v_j)_{1 \leq i,j \leq d}.  \]
Then, we call \textit{divergence} of $\sigma$ the vector given by:
\[  \Div( \sigma) = \left( \sum_{i=1}^d \frac{\partial \sigma_{ij}}{\partial x_j} \right)_{1 \leq i,j \leq d} .\]
Finally, we call \textit{scalar product} of $\sigma$ and $\tau$ the real function:
\[ \sigma : \tau = \sum_{ 1 \leq i,j \leq d} \sigma_{ij} \tau_{ij}. \]
Note that by definition we have $\sigma : \tau = \sigma^\top : \tau^\top$, and the norm associated to this scalar product is simply denoted by $| \cdot |$ in such a way that $|\sigma|^2= \sigma : \sigma$. We now give some useful properties from differential calculus:
\begin{proposition}
Let  $\rho$ be a scalar and $u$, $v$, $w$ be three vectors smooth enough on $\Omega$, then the following holds:
\begin{itemize}
    \item[$\bullet$] $(u \otimes v)w = (v \cdot w)u$,
    \item[$\bullet$] $\Div(u \otimes v) = (\Div \ v)u + (v \cdot \nabla)u$,
    \item[$\bullet$] $\Div( \rho u)= \nabla \rho \cdot u + \rho \ \Div \ u$,
    \item[$\bullet$] $\Div( \rho u \otimes v)= (\nabla \rho \cdot v)u + \rho (v \cdot \nabla)u + \rho \ \Div(v)u$.
\end{itemize}
\end{proposition}

We now state some technical lemmas which are used throughout this paper.
\begin{lemma} \label{lemma_power}
Let $T>0$ and $p,q \geq 1$, $g \in L^{\infty}(0,T;L^q(\T^d))$, and assume $\alpha:=\frac{pq}{p+q} \geq 1$. Then,
\begin{itemize}
    \item  $f \in L^2(0,T;L^p(\T^d)) \Rightarrow fg \in L^2(0,T;L^{\alpha}(\T^d))$,
    \item $f \in L^{\infty}(0,T;L^p(\T^d)) \Rightarrow fg \in L^{\infty}(0,T;L^{\alpha}(\T^d))$.
\end{itemize}
\end{lemma}

For a proof of the following lemma, see \cite{jungel2010}:
\begin{lemma} \label{technical_lemma}
We have the following inequalities :
\[ \int_{\T^d} \rho |\nabla^2 \log\rho |^2 dx \geq  \kappa_d \int_{\T^d} | \nabla^2 \sqrt{\rho} |^2 dx \]
with $\kappa_2=7/8$ and $\kappa_3=11/15$, and
\[ \int_{\T^d} \rho |\nabla^2 \log\rho |^2 dx \geq  \kappa  \int_{\T^d} | \nabla \sqrt[4]{\rho} |^4 dx,  \]
with $\kappa > 0$.
\end{lemma}

\begin{lemma}\textbf{(Gagliardo-Nirenberg inequality).} \\
Let $1\leq p,q,r \leq \infty$ and let $j$, $m$ be two integers, $0\leq j < m$. If
\[ \frac{1}{p} = \frac{j}{d} + \alpha \left( \frac{1}{r} - \frac{m}{d}   \right) + \frac{1-\alpha}{q}  \]
for some $a \in \left[ j/m,1 \right]$ ($a<1$ if $r>1$ and $m-j-d/r=0$), then there exists $C=C(d,m,j,a,q,r)$ such that
\[ \| \D^j u \|_{L^p(\T^d)}  \leq C \| \D^j u \|_{L^r(\T^d)}^{\alpha} \| u \|_{L^q(\T^d)}^{(1-\alpha)}\]
for all $u \in W^{m,j}(\T^d) \cap L^q(\T^d)$.
\label{gagliardo_nirenberg}
\end{lemma}

For a proof of the following lemma, see \cite{carles2019}:
\begin{lemma} \label{inverse_lemma}
For $n \in \mathbb{N}^*$, there holds
\[  \| \nabla^n(f^{-1}) \|_{L^2(\T^d)} \lesssim (1+\| f^{-1} \|_{L^4(\T^d)} + \| f^{-1} \|_{L^{2(n+1)}(\T^d)})^{n+1}(1+ \| f \|_{H^{\sigma}(\T^d)} )^n  \]
with $\sigma > n +d/2$.
\end{lemma}

\bibliographystyle{plain}
\bibliography{biblio}

\begin{thebibliography}{10}

\bibitem{ineg_sobo_log}
C{\'e}cile An{\'e}, S{\'e}bastien Blach{\`e}re, Djalil Chafai, Pierre
  Foug{\`e}res, Ivan Gentil, Florent Malrieu, Cyril Roberto, and Gr{\'e}gory
  Scheffer.
\newblock {\em {Sur les in{\'e}galit{\'e}s de Sobolev logarithmiques}}.
\newblock {Soci{\'e}t{\'e} Math{\'e}matique de France}, 2000.

\bibitem{antonelli2009}
Paolo Antonelli and Pierangelo Marcati.
\newblock On the finite energy weak solutions to a system in quantum fluid
  dynamics.
\newblock {\em Comm. Math. Phys.}, 287(2):657--686, 2009.

\bibitem{bresch2015}
Didier Bresch, Fr{\'e}d{\'e}ric Couderc, Pascal Noble, and Jean-Paul Vila.
\newblock {A generalization of the quantum Bohm identity: Hyperbolic CFL
  condition for Euler--Korteweg equations}.
\newblock {\em {Comptes Rendus Math{\'e}matique}}, November 2015.

\bibitem{bresch2003}
Didier Bresch, Beno\^{\i}t Desjardins, and Chi-Kun Lin.
\newblock On some compressible fluid models: {K}orteweg, lubrication, and
  shallow water systems.
\newblock {\em Comm. Partial Differential Equations}, 28(3-4):843--868, 2003.

\bibitem{bresch2004}
Didier Bresch and Benoît Desjardins.
\newblock Quelques modèles diffusifs capillaires de type {K}orteweg.
\newblock {\em Comptes Rendus Mécanique}, 332(11):881 -- 886, 2004.

\bibitem{Bresch2019}
Didier Bresch, Marguerite Gisclon, and Ingrid Lacroix-Violet.
\newblock On {N}avier--{S}tokes--{K}orteweg and {E}uler--{K}orteweg systems:
  Application to quantum fluids models.
\newblock {\em Archive for Rational Mechanics and Analysis}, 233(3):975--1025,
  Sep 2019.

\bibitem{bresch2020}
Didier Bresch, Marguerite Gisclon, Ingrid Lacroix-Violet, and Alexis Vasseur.
\newblock On the exponential decay for compressible
  {N}avier-{S}tokes-{K}orteweg equations with a drag term.
\newblock Preprint, archived at \url{https://arxiv.org/pdf/2004.07895.pdf}.

\bibitem{carles2019}
R{\'e}mi Carles, Kleber Carrapatoso, and Matthieu Hillairet.
\newblock {Global weak solutions for quantum isothermal fluids}.
\newblock May 2019.
\newblock Preprint, archived at \url{https://arxiv.org/pdf/1905.00732.pdf}.

\bibitem{carles2018}
R\'{e}mi Carles and Isabelle Gallagher.
\newblock Universal dynamics for the defocusing logarithmic {S}chr\"{o}dinger
  equation.
\newblock {\em Duke Math. J.}, 167(9):1761--1801, 2018.

\bibitem{chauleur2020}
Quentin Chauleur.
\newblock {Dynamics of the Schr{\"o}dinger-Langevin equation}.
\newblock Preprint, archived at \url{https://arxiv.org/pdf/2004.06962.pdf}.

\bibitem{chavanis2017}
Pierre-Henri Chavanis.
\newblock {Derivation of a generalized Schr{\"o}dinger equation from the theory
  of scale relativity}.
\newblock {\em {Eur.Phys.J.Plus}}, 132(6):286, 2017.

\bibitem{chavanis2019cosmo}
Pierre-Henri Chavanis.
\newblock Derivation of the core mass-halo mass relation of fermionic and
  bosonic dark matter halos from an effective thermodynamical model.
\newblock {\em Physical Review D}, 100(12), Dec 2019.

\bibitem{chavanis2019stat}
Pierre-Henri Chavanis.
\newblock {Generalized Euler, Smoluchowski and Schr{\"o}dinger equations
  admitting self-similar solutions with a Tsallis invariant profile}.
\newblock {\em {Eur.Phys.J.Plus}}, 134(7):353, 2019.

\bibitem{feireisl2014}
Eduard Feireisl.
\newblock Relative entropies, dissipative solutions, and singular limits of
  complete fluid systems.
\newblock In {\em Hyperbolic problems: theory, numerics, applications},
  volume~8 of {\em AIMS Ser. Appl. Math.}, pages 11--27. Am. Inst. Math. Sci.
  (AIMS), Springfield, MO, 2014.

\bibitem{ferriere2020}
Guillaume Ferriere.
\newblock {W}{K}{B} analysis and semiclassical limit of the logarithmic
  non-linear {S}chr\"{o}dinger equation in an analytic framework, 2020.
\newblock Work in progress.

\bibitem{jungel2010}
Ansgar Jüngel.
\newblock Global weak solutions to compressible {N}avier–{S}tokes equations
  for quantum fluids.
\newblock {\em SIAM J. Math. Analysis}, 42:1025--1045, 01 2010.

\bibitem{vasseur2018}
Ingrid Lacroix-Violet and Alexis Vasseur.
\newblock Global weak solutions to the compressible quantum {N}avier-{S}tokes
  equation and its semi-classical limit.
\newblock {\em J. Math. Pures Appl. (9)}, 114:191--210, 2018.

\bibitem{lions1996}
Pierre-Louis Lions.
\newblock {\em Mathematical topics in fluid mechanics. {V}ol. 1}, volume~3 of
  {\em Oxford Lecture Series in Mathematics and its Applications}.
\newblock The Clarendon Press, Oxford University Press, New York, 1996.
\newblock Incompressible models, Oxford Science Publications.

\bibitem{mousavi2019}
S.~V. Mousavi and S.~Miret-Artés.
\newblock On non-linear {S}chrödinger equations for open quantum systems.
\newblock {\em The European Physical Journal Plus}, 134(9), Sep 2019.

\bibitem{nassar1985}
A~B Nassar.
\newblock Fluid formulation of a generalised {S}chr{\"o}dinger-{L}angevin
  equation.
\newblock {\em Journal of Physics A: Mathematical and General},
  18(9):L509--L511, jun 1985.

\bibitem{nassar}
Antonio Nassar and Salvador Miret-Artés.
\newblock {\em Bohmian Mechanics, Open Quantum Systems and Continuous
  Measurements}.
\newblock 01 2017.

\bibitem{vasseur2016}
Alexis~F. Vasseur and Cheng Yu.
\newblock Global weak solutions to the compressible quantum {N}avier-{S}tokes
  equations with damping.
\newblock {\em SIAM J. Math. Anal.}, 48(2):1489--1511, 2016.

\bibitem{zander}
C.~Zander, A.~R. Plastino, and J.~D\'{\i}az-Alonso.
\newblock Wave packet dynamics for a non-linear {S}chr\"{o}dinger equation
  describing continuous position measurements.
\newblock {\em Ann. Physics}, 362:36--56, 2015.

\end{thebibliography}

\end{document}